\documentclass[11pt]{article}

\usepackage[margin=1in]{geometry}
\usepackage[T1]{fontenc}
\usepackage{lmodern}
\usepackage{microtype}
\usepackage{amsmath,amssymb,amsthm,mathtools}
\usepackage{booktabs,array}
\usepackage{enumitem}
\usepackage{xcolor}
\usepackage{float}
\usepackage{tikz}
\usepackage{comment}
\usetikzlibrary{arrows.meta,calc,positioning,shapes.geometric}
\usepackage[hidelinks]{hyperref}

\usepackage{adjustbox}

\hypersetup{
  pdftitle={Improved l0-Isoperimetry for Convex Bodies via Mass Transport},
  pdfauthor={Manuel Fernandez V}
}

\allowdisplaybreaks
\setlist{itemsep=2pt,topsep=4pt}
\numberwithin{equation}{section}

\newtheorem{theorem}{Theorem}[section]
\newtheorem{proposition}[theorem]{Proposition}
\newtheorem{lemma}[theorem]{Lemma}

\theoremstyle{definition}
\newtheorem{definition}[theorem]{Definition}
\newtheorem{remark}[theorem]{Remark}

\newcommand{\R}{\mathbb{R}}

\newcommand{\poly}{\operatorname{poly}}
\newcommand{\vol}{\operatorname{vol}}

\newcommand{\E}{\mathbb{E}}
\newcommand{\Id}{\operatorname{Id}}
\newcommand{\one}{\mathbf{1}}
\newcommand{\defeq}{\mathrel{:=}}
\newcommand{\eps}{\varepsilon}
\newcommand{\sgn}{\operatorname{sign}}
\newcommand{\ind}{\textbf{1}}
\newcommand{\bren}{\nabla\varphi}
\newcommand{\hess}{\nabla^2\varphi}
\newcommand{\ip}[2]{\langle #1,#2\rangle}

\DeclarePairedDelimiter{\floor}{\lfloor}{\rfloor}
\DeclarePairedDelimiter{\ceil}{\lceil}{\rceil}
\newcommand{\norm}[1]{\lVert #1\rVert_2}
\newcommand{\normQ}[1]{\lVert #1\rVert_Q}

\title{Improved $\ell_0$-Isoperimetry for Convex Bodies via Mass Transport}
\author{Manuel Fernandez V}
\date{August 2026}

\begin{document}
\maketitle

\begin{abstract}
We study $\ell_0$ isoperimetry for a convex body $K\subset\R^n$,
$n\ge2$.  For a Borel set $S\subset K$, let $\partial_0^K S$ be the set of
points in $K\setminus S$ that can be reached from $S$ by changing at
most one coordinate (i.e. the $\ell_0$ boundary of $S$).  Suppose that, for some unconditional convex body
$Q\subset\R^n$, numbers $r,R>0$, and possibly different centers
$x_0,y_0$,
\[
 x_0+rQ\subset K\subset y_0+RQ.
\]
Writing $s=\vol(S)/\vol(K)$, we prove that whenever $0<s\le1/2$,
\[
 \frac{\vol(\partial_0^K S)}{\vol(S)}
 \ge \frac{cr}{nR}
 \min\left\{1,\frac{\log(e/s)}{n}\right\},
\]
where $c > 0$ is an absolute constant.
Consequently, the associated $\ell_0$-isoperimetric coefficient is at least
$cr/(n^2R)$. 

Previously, lower bounds for regularities other than $\ell_2$ and $\ell_\infty$ were obtained through sandwiching-type arguments and lost up to a factor of $n^{1/2}$ or $n$ in comparison to the corresponding lower bound in $\ell_2$ or $\ell_\infty$. In contrast our lower bound holds directly for any $Q$-regularity, where $Q$ is an unconditional convex body. This includes $\ell_p$-regularity for every $1\le p\le\infty$. Compared to $\ell_2$ and $\ell_\infty$ regularity, our lower bound result improves upon the previously best known lower bounds, for any $s$, by a factor of $n$.  As an application of our result, we give improved mixing time bounds for the Coordinate Hit and Run walk (CHAR).

Our proof of the lower bound is based on the method of canonical paths, a tool typically used to prove discrete/combinatorial isoperimetric inequalities on graphs. 
At a high level the method asks us to find, for any vertex subset $S$ of $G$, many paths from $S$ to $S^c$ where no vertex/edge is covered by too many paths. 
Because these paths cross the vertex/edge boundary of $S$, their low vertex/edge congestion imply that the boundary is large. 
In our setting the corresponding graph on $K$ has a continuum number of vertices and edges and the notion of size is volume, so an appropriate formulation is necessary. Our construction of canonical paths can be viewed as a suitable coordinate discretization of certain mass transport maps from $S$ to $S^c$.

To complement our lower bounds we give two
upper bound constructions for any $Q$-regularity. We show that slanted $Q$-cylinders have constant-sized subsets with $\ell_0$ boundary having relative volume equal to $O(r/(Rn))$. We also show that slanted $Q$-cones have exponentially small subsets with $\ell_0$ boundary having relative volume equal to $O(r/R)$. Both constructions only require that $R \ge Cr$, for some absolute constant $C$. For constant-sized and exponentially small-sized sets, 
our upper and lower bounds are within a factor of $O(n)$ of each other.
We conjecture that under unconditional $Q$-regularity the correct
constant-set scale is of order $r/(nR)$ and the exponentially small-set
scale is of order $r/R$.

\end{abstract}

\newpage
\section{Introduction}

Gibbs sampling is one of the basic coordinatewise methods for sampling a
multivariate distribution.  It goes back at least to Turchin's Monte Carlo
integration method and was later popularized in statistical imaging by Geman
and Geman~\cite{Turchin1971,GemanGeman1984}.  At each step one chooses a
coordinate, holds all other coordinates fixed, and resamples the chosen
coordinate from its conditional distribution.  For the uniform distribution on a convex body
$K\subset\R^n$, this is coordinate hit-and-run (CHAR): from $x\in K$, choose
$i\in[n]$ uniformly and sample the next point uniformly from the
\emph{coordinate fiber} $K\cap(x+\R e_i)$.  We call a transition between two
points on a common coordinate fiber a \emph{coordinate move}; equivalently,
the two points differ in at most one coordinate.  Unlike ordinary hit-and-run, which chooses a random
direction, CHAR moves only in the standard coordinate directions.  This can
make each update substantially cheaper, especially for polyhedral bodies, but
it also makes the geometry of convergence sensitive to the chosen coordinate
system~\cite{LaddhaVempala2023,NarayananSrivastava2022}.

The boundary notion adapted to these moves is measured by volume rather than
surface area.  For a Borel set $S\subset K$, let $\partial_0^K S$ be the set of
points in $K\setminus S$ that can be reached from $S$ by changing at most one
coordinate. Equivalently, $\partial_0^K S$ is the exterior vertex boundary
in the continuous graph over $K$ that joins points on a common coordinate fiber. We refer to $\partial_0^K S$ as the \emph{$\ell_0$ boundary of $S$}.
Two sets are \emph{axis-disjoint} if no such edge joins them.
In this paper we study the one-set coefficient $\xi_0(K)$, and the separator coefficient $\psi_0(K)$;
their definitions and their equivalence up to the normalization in
Proposition~\ref{prop:equivalence} are given in
Section~\ref{sec:prelim}.

A convex body
$Q\subset\R^n$ is \emph{unconditional} if it is invariant under
arbitrary sign changes of its coordinates.  We assume that, for possibly
different centers $x_0,y_0$,
\begin{equation}\label{eq:unconditional-regularity}
 x_0+rQ\subset K\subset y_0+RQ.
\end{equation}
The ratio $R/r$ is invariant under scalar dilation, and both coordinate
coefficients are invariant under translations and scalar dilations.  Taking
$Q=B_p^n$ gives the $\ell_p$ setting for every $1\le p\le\infty$. Note that
orientation of $Q$ matters: arbitrary affine changes of coordinates do not
preserve $\ell_0$ boundary.

An initial law $\mu_0$ is \emph{$M$-warm} with respect to the target law
$\pi$ if $\mu_0(A)\le M\pi(A)$ for every measurable $A$.  We call $\mu_0$ a
\emph{warm start} when $M = O(\poly n)$ and a \emph{cold start} when $M = O(\exp(\poly n))$.

\subsection{$\ell_0$ isoperimetry and polynomial mixing time of CHAR}
The first polynomial mixing guarantees for CHAR on general convex bodies were
proved concurrently by Laddha and Vempala and by Narayanan and
Srivastava~\cite{LaddhaVempala2023,NarayananSrivastava2022}.  Laddha and
Vempala assume that $K$ contains a Euclidean ball of radius $r$ and use the
root-mean-square radius
\[
 R_2^2=\mathbb E_{X\sim\operatorname{Unif}(K)}
       \|X-z_K\|_2^2
\]
about the centroid.  From an $M$-warm start they obtain, up to logarithmic
factors, the mixing bound
\[
 \widetilde O\!\left(
   M^2n^9(R_2/r)^2\varepsilon^{-2}
 \right).
\]
To prove this bound they gave a lowerbound for the size of the separator of axis-disjoint sets.
More specifically, under the above normalization, they showed that if $K=S_1 \cup S_2
\cup S_3$, with $S_1,S_2$ axis-disjoint, then for every  $\eps>0$ one has 
\begin{equation}\label{eq:LV-intro}
 \vol(S_3)\ge
 \frac{c\eps r}{n^{7/2}R_2}
 \left(\min\{\vol(S_1),\vol(S_2)\}-\eps\vol(K)\right).
\end{equation}
The subtractive term is the cost of discarding a boundary layer before a
fixed-scale cube decomposition is applied.

The Laddha--Vempala proof begins by first establishing a stronger separator inequality
 when $K$ is an axis-aligned cube. It then considers a tiling of the deep part of $K$ by small cubes.
If much of the smaller set lies in locally mixed cubes, the separator inequality for the cube
implies a non-small separator for $K$. Otherwise, the union of cubes for which that small set has
high density has a large Euclidean boundary and that, coupled with an analysis of the adjacent cubes, shows that adjacent cubes contain a non-small separator for $K$.  A one-step coupling then converts the separator inequality into a $s$-conductance bound.

The work of Narayanan and Srivastava work with the $\ell_\infty$ regularization $B_\infty^n\subset K\subset R_\infty B_\infty^n$.
They prove, from an $M$-warm start, a mixing time bound of the form
\[
 \widetilde O\!\left(
 M^4n^7R_\infty^4\varepsilon^{-4}
 \right)
\]
~\cite{NarayananSrivastava2022}.  Their mechanism is different from
cube tiling.  They introduce a coordinate Gaussian proposal chain, run it for
$O(n\log n)$ steps so that nearby starting points have overlapping endpoint
laws, and compare its flow with that of CHAR. Their norm distance and the
cross-ratio/Hilbert-metric isoperimetry of Lov\'asz--Simonovits then turn this
multistep overlap into an $s$-conductance estimate. As in the
Laddha--Vempala argument, a boundary layer must first be removed and the resulting $s$-conductance bound can only give polynomial time mixing from a warm start.

Narayanan, Rajaraman, and Srivastava subsequently removed this obstruction by
using a Whitney-type multiscale decomposition~\cite{NarayananRajaramanSrivastava2025}.
Under $r_\infty B_\infty^n\subset K\subset R_\infty B_\infty^n$,
they prove the cold-start bound
\[
 O\!\left(n^9(R_\infty/r_\infty)^2
             \log(M/\varepsilon)\right).
\]
Their axis-disjoint separator theorem has the form
\[
 \vol(S_3)\ge c\frac{\Phi_{\mathcal M_\infty}}{n^{3/2}}
       \min\{\vol(S_1),\vol(S_2)\},
\]
where $\Phi_{\mathcal M_\infty}$ is the conductance of their Whitney-cube
chain.  They show that
$\Phi_{\mathcal M_\infty}\gtrsim(r_\infty/R_\infty)n^{-2}$, which gives the
separator scale $(r_\infty/R_\infty)n^{-7/2}$. Unlike the tiling used in Laddha-Vempala, their cubes shrink as they approach $\partial K$, and adjacent cubes
have comparable scales.  Analytically this corresponds to a degenerate
Finsler metric that magnifies distances by the reciprocal of the distance to
the boundary.  The metric isoperimetry and multiscale conductance control
axis-disjoint sets at every volume scale instead of throwing away a fixed
boundary layer.

In related work~\cite{Fernandez2026}, the author proved an optimal separator theorem for axis-disjoint sets on the cube, showing that
\[
 \psi_0([0,1]^n)=\Theta(n^{-1/2})
\]
and proved the universal upper bound
\[
 \psi_0(K)=O(n^{-1/2})
\]
for every measurable finite-volume $K\subset\R^n$.  The lower bound on the
cube uses coordinate shaking to reduce arbitrary axis-disjoint sets to a
compressed form to which Harper's vertex-isoperimetric theorem for Hamming graphs applies.  It
improves the cube input in the earlier CHAR analyses by a factor of $n^{1/2}$
and the associated mixing bounds by a factor $n$.  The universal upper bound
also shows that this local reduction to cube isoperimetry is, as a method for
general bodies, essentially tight: no uniformly stronger local cube
coefficient is available.

In later work,
Wadia extended the axis-disjoint framework of Laddha-Vempala beyond that of uniform measures on convex bodies. She considered log-concave measures on $\R^n$ with densities $e^{-V}$ for which $V$ is
$\mu$-strongly convex and $L$-smooth, with condition number
$\kappa=L/\mu$~\cite{Wadia2024}. She showed that the mixing bound, from an $M$-warm start, the resulting bound
is
\[
 O^*\!\left(
   \kappa^2n^{15/2}
   \max\left\{1,\frac1n\log\frac{2M}{\varepsilon}\right\}
 \right).
\]
The proof is based off localizing to a high-probability Euclidean ball, tiling it at a scale
on which log-smoothness makes the density nearly constant, applying the sharp
cube separator inequality and converting the weighted axis-disjoint separator estimate into
an $s$-conductance bound.  Thus the result retains the Laddha--Vempala
architecture, with concentration and local density comparison replacing the
hard boundary of a convex body.

We now mention two complementary lines of work on Gibbs samplers for log-concave measures. 
 Ascolani, Lavenant, and Zanella prove sharp relative-entropy
contraction for random-scan Gibbs sampling under strong log-concavity and
coordinate smoothness~\cite{AscolaniLavenantZanella2024}.  For $N$ coordinate
blocks and condition number $\kappa$, their basic estimate contracts entropy
by the factor $1-1/(\kappa N)$ per update, leading to
$O(\kappa N\log(1/\varepsilon))$ updates.  Their proof uses the variational
structure of conditional resampling together with triangular
Knothe--Rosenblatt transport maps to compare simultaneous and
one-coordinate motion.  More recently, Goyal,
Deligiannidis, and Kantas derived conductance bounds for random-scan and
systematic-scan Gibbs samplers from Poincar\'e or log-Sobolev inequalities and
regularity of the conditional distributions; their proof uses new
three-set isoperimetric inequalities and sequential coupling~\cite{GoyalDeligiannidisKantas2025}. Neither results are directly comparable with uniform sampling from a
convex body and it is not clear if there methods can be used to recover our results.

\subsection{Improved $\ell_0$ isoperimetry and mixing time bounds}
Our main result is an improved $\ell_0$ isoperimetric inequality for convex sets satisfying unconditional
$Q$-regularity.

\begin{theorem}[Global lower bound and small-set expansion]\label{thm:main}
There is a universal constant $c>0$ such that, for every $n\ge2$, every
$r,R>0$, every unconditional convex body $Q\subset\R^n$, every pair of
centers $x_0,y_0$, and every convex body $K$ satisfying
$x_0+rQ\subset K\subset y_0+RQ$, every Borel set $S\subset K$ with
$0<\vol(S)\le\vol(K)/2$ satisfies, with
$s=\vol(S)/\vol(K)$ its relative volume,
\begin{equation}\label{eq:set-main}
 \frac{\vol(\partial_0^K S)}{\vol(S)}
 \ge \frac{cr}{nR}
 \min\left\{1,\frac{\log(e/s)}{n}\right\}
 =
 \begin{cases}
  \displaystyle \frac{cr}{n^2R}\log\frac es,
       & s \ge e^{1-n},\\[2mm]
  \displaystyle \frac{cr}{nR},
       & s \le e^{1-n},
 \end{cases}
\end{equation}
Consequently,
\begin{equation}\label{eq:main-bound}
 \frac{cr}{n^2R}\le \xi_0(K)\le\psi_0(K).
\end{equation}
\end{theorem}

In the proof of Theorem \ref{thm:main} we assume that $K$ has volume 1. This can be done without loss of generality since the relative size of the $\ell_0$ boundary and the ratio of the regularity parameters are invariant under scaling. We use this normalization for every ambient convex body in
Sections~\ref{sec:prelim} and~\ref{sec:lower}. Assuming it, $\vol(S)$ then equals the
relative volume of $S$.  The volume-one definitions in
Section~\ref{sec:prelim} therefore also define $\xi_0$ and $\psi_0$
canonically for an arbitrary unnormalized body by dilation. We also note that the condition $s \le 1/2$ is not essential and that a straightforward, although tedious, modification of the proof shows that the theorem holds so long as $s$ is bounded away from 1 by an absolute constant.

We call a set balanced when its relative volume stays bounded away
from both zero and one.  The balanced-set scale is the order of the
boundary volume in that regime.  Its dependence on $s$ as $s\downarrow0$ is
the \emph{small-set profile}; an improvement over the balanced-set scale that
grows as $s$ decreases is a small-set gain, and the gain saturates
once further decreasing $s$ no longer improves the bound.

The regularity class in Theorem~\ref{thm:main} includes every $\ell_p$
regularity class, $1\le p\le\infty$, and permits different centers for the inner and
outer $Q$-copies.  Under Euclidean regularity, substituting the sharp cube
theorem from~\cite{Fernandez2026} into the Laddha--Vempala argument gives the
balanced-set scale
$c(r/R)n^{-3}$.  Under $\ell_\infty$-regularity, the corresponding sharp-cube
refinement of the multiscale result above gives the same
$c(r_\infty/R_\infty)n^{-3}$ scale.  Theorem~\ref{thm:main} gives
$c(r/R)n^{-2}$ in either matching normalization, an improvement by a factor
of $n$ over these refined bounds.  It also gives a direct small-set gain:
as $s$ decreases, the $\ell_0$ boundary grows by the factor $\log(e/s)$ until
$s\le e^{1-n}$, where it saturates at $cr/(nR)$. Narayana, Srivastava, and Rajaraman also implicitly derive a small-scale profile for their lower bound on $\vol(S_3)$ under $\ell_\infty$ regularity (see \cite[Theorem 5.1, Theorem 5.3, Theorem 6.2]{NarayananRajaramanSrivastava2025}), but our bound is better by a factor of $n$. 

Using Theorem \ref{thm:main} we are able to deduce improved mixing time bounds for CHAR. This is recorded across two remarks. The first remark explains how Theorem \ref{thm:main} holds for an effectively smaller outer regularity parameter, depending on the size of the subset. The second remark records the statements for the improved mixing time for CHAR and their derivations.
\begin{remark}[Effective regularity under isotropicity]
\label{rem:affine-normalization}
As observed by Laddha and Vempala \cite{LaddhaVempala2023}, every full-dimensional convex body $K$ can
be put in isotropic position by an invertible affine map. The normalized body $\widetilde{K}$ has centroid 0 and covariance matrix $I$.  Consequently,
for $X$ uniform on $\widetilde K$,
\[
 R_2^2=\mathbb E\|X\|_2^2=n.
\]
Morever, every isotropic convex body admits the following $\ell_2$ regularity
\cite[Theorem 4.1]{KannanLovaszSimonovits1995}:
\begin{equation}\label{eq:isotropic-deterministic-sandwich}
 \sqrt{\frac{n+2}{n}}B_2^n
 \subset\widetilde K
 \subset\sqrt{n(n+2)}B_2^n.
\end{equation}
In particular, $B_2^n\subset\widetilde K\subset(n+1)B_2^n$.

Although Theorem~\ref{thm:main} assumes an outer containment rather than
a moment bound, Paouris' large-deviation inequality \cite{Paouris2006} recovers the same
$R_2$ dependence and, moreover, gives a useful truncation at smaller volume
scales.  In the present normalization it states that there is a universal
$C>0$ such that, for every $t\ge1$,
\begin{equation}\label{eq:paouris-tail}
 \mathbb P\!\left(\|X\|_2>Ct\sqrt n\right)
 \le e^{-t\sqrt n}
 \qquad (X\sim\operatorname{Unif}(\widetilde K)).
\end{equation}
Set $K_t=\widetilde K\cap Ct\sqrt n B_2^n$.
If $s=\vol(S)/\vol(\widetilde K)\le1/2$ and
\begin{equation}\label{eq:paouris-admissible-scale}
 s\ge e^{-t\sqrt n/4},
\end{equation}
then \eqref{eq:paouris-tail} shows that the discarded volume is at most
$s^2e^{-\sqrt{n}/2}\vol(\widetilde K)$.  Thus $1 \ge \vol(S\cap K_t)/\vol(S) \ge 1 - (1/(2e^{1/2})) \ge 0.69$ and $1 \ge \vol(K_t \cap S)/\vol(S) \ge 1/2$. Writing $\tilde{s} = \vol(K_t \cap S)/\vol(K_t)$ we get that $\tilde{s}/s \in [0.5,1.6]$.
 Now observe that $\vol(\partial_0^{\widetilde K}S) \ge \vol(\partial_0^{K_t} (S \cap K_t))$ and $K_t$ still contains $B_2^n$. Applying the separator form of
Theorem~\ref{thm:main} at relative scale $\tilde{s}$, which is comparable to $s$, gives
\begin{equation}\label{eq:paouris-isoperimetric-profile}
 \frac{\vol(\partial_0^{\widetilde K}S)}{\vol(S)}
 \ge
\frac{c \vol(\partial_0^{K_t}(S \cap K_t))}{\vol(S \cap K_t)}
\ge
 \frac{c}{t n^{3/2}}
       \min\left\{1,\frac{\log(e/s)}{n}\right\}.
\end{equation}
Note that the application is valid since the lower bound in Theorem~\ref{thm:main} is the same when $s$ is scaled by an absolute constant (with a different $c$) and $\tilde{s} \le 0.9$ (so that Theorem \ref{thm:main} applies)
For sets of constant relative volume one may take $t=1$.
More generally, Paouris' inequality allows the outer radius
$R=Ct\sqrt n$ while preserving a constant fraction of every set as small
as $e^{-t\sqrt n/4}$.  This truncation is needed only for
$1\le t\lesssim\sqrt n$: when $t\asymp\sqrt n$, its radius is already
$O(n)$, and
\eqref{eq:isotropic-deterministic-sandwich} places all of $\widetilde K$ in
such a ball.  Thus below relative volume $e^{-cn}$ one simply applies
Theorem~\ref{thm:main} to the whole body with outer radius $O(n)$.
\end{remark}

\begin{remark}[Small-set conductance and cold starts]
\label{rem:conductance-profile}
Let $P_x$ denote the one-step transition law of lazy CHAR
from $x$, and, for $0<s\le1/2$, define
\[
 \Phi_{\mathrm{CHAR}}(s)
 :=\inf_{\substack{A\subset K\ \mathrm{measurable}\\
                    0<\vol(A)\le s\vol(K)}}
    \frac{1}{\vol(A)}\int_A P_x(K\setminus A)\,dx.
\]
Applying the
one-step coupling \cite[Lemma 4]{LaddhaVempala2023} and using
\eqref{eq:set-main}, gives
\begin{equation}\label{eq:char-conductance-profile}
 \Phi_{\mathrm{CHAR}}(s)
 \ge \frac{cr}{n^2R}
       \min\left\{1,\frac{\log(e/s)}{n}\right\}.
\end{equation}
The average-conductance integral is bounded by
\[
 \int_{1/M}^{1/2}\frac{ds}{s\Phi_{\mathrm{CHAR}}(s)^2}
 \le C\left(\frac Rr\right)^2
       \left[n^6+n^4\max\{\log M-n,0\}\right].
\]
For the uniform start on $x_0+rQ$, where $M\le(R/r)^n$, the
average-conductance bound \cite[Theorem~2.2]{LovaszKannan1999} gives 
\[
 T_{\mathrm{mix}}^{\mathrm{CHAR}}(\varepsilon)
 \le C\left(\frac Rr\right)^2
 \left[n^6+n^5\log(R/r)\right]
 \log(1/\eps).
\]

We finally specialize to a convex body $K$ in isotropic position (see Remark~\ref{rem:affine-normalization}), with CHAR initialized from the
uniform measure on $B_2^n\subset K$. We consider this initialization because a convex body can be brought to isotropic position using an affine transformation that is determined from the body's centroid and covariance matrix, both of which can be estimated relatively efficiently and therefore may be viewed as a one-time preprocessing step, and producing a sample in the isotropic position within $\eps$ TV-distance of uniform gives a sample in the original position within $\eps$ TV-distance of uniform after applying the inverse transformation to it.

For a volume scale $s$, put
$L_s=\log(e/s)$.  As long as $L_s\le n/2$, choose
$t_s=\max\left\{1,\frac{4L_s}{\sqrt n}\right\}$. Then \eqref{eq:paouris-admissible-scale} holds and $t_s\le 2\sqrt n$.
For $L_s>n/2$, instead apply Theorem~\ref{thm:main} directly to $K$ using
\eqref{eq:isotropic-deterministic-sandwich}.  Combining the two regimes, we
may use the effective radius
$\rho(s):=\min\bigl\{n,\max\{\sqrt n,L_s\}\bigr\}$.

Combining \eqref{eq:paouris-isoperimetric-profile} with the one-step CHAR
coupling gives the scale-dependent estimate
\begin{equation}\label{eq:isotropic-char-profile}
 \Phi_{\mathrm{CHAR}}(s)
 \ge \frac{c}{n^2\rho(s)}
       \min\left\{1,\frac{L_s}{n}\right\}.
\end{equation}
The unit-ball start is $M$-warm with
$M=\frac{\vol(K)}{\vol(B_2^n)}$.
\eqref{eq:isotropic-deterministic-sandwich} implies that 
 $\log M\le Cn\log(en)$.

Writing the average-conductance integral in the variable
$L=\log(e/s)$, \eqref{eq:isotropic-char-profile} grows up to
$L=\sqrt n$ and then stays at order $n^{-3}$.  Consequently,
\[
 \int_{1/M}^{1/2}
   \frac{ds}{s\Phi_{\mathrm{CHAR}}(s)^2}
 \le C\left(n^7+n^6\log(eM)\right)
 \le Cn^7\log(n).
\]
Thus, the best bound furnished by the Theorem \ref{thm:main} for the mixing time of CHAR starting from this
initialization is
\begin{equation}\label{eq:isotropic-unit-ball-mixing}
 T_{\mathrm{mix}}^{\mathrm{CHAR}}(\varepsilon)
 \le O(n^7\log(n)\log(1/\eps)).
\end{equation}
If one instead modified the above argument to use $\ell_{\infty}$ and instead used the implied small set profile from \cite{NarayananRajaramanSrivastava2025} the mixing time from this initialization would be bounded by $\tilde{O}(n^{10}\log(1/\eps))$. 
\end{remark}

\subsection{An improved upper bound construction}
To complement Theorem \ref{thm:main} we show the existence of a convex body with subsets having $\ell_0$ boundary with relative volume within a factor of $n$ of the lower bound given by Theorem \ref{thm:main} in the constant-sized and exponentially-small sized regimes. The subsets realizing the upper bounds are based on the sign-count mechanism used to prove upper bounds for general convex bodies in \cite{Fernandez2026}.
To begin we prove that subsets of the desired type exist either in $Q$-cones or in $Q$-cylinders. See Section \ref{sec:upper} for the definition of these convex bodies.

\begin{proposition}[$Q$-cylinder]\label{prop:capsule}
There are universal constants $C_0,C,c>0$ such that the following
holds.  Let $n\ge 5$, let $Q\subset\R^n$ be unconditional, let $r>0$ and $L\ge C_0r$. The body $K = K_{L,r}^Q$ contains $(r/2)Q$  and is contained in $(L+r)Q$.
In addition there exist axis-disjoint sets $A,B\subset K_{L,r}^Q$ such that
\[
 \min\{\vol(A),\vol(B)\}\ge c\vol(K_{L,r}^Q) \qquad 
 \frac{\vol(K\setminus (A\cup B))}{\min\{\vol(A),\vol(B)\}} \le \frac{Cr}{nL}.
\]
Thus, in terms of its inner and outer $Q$-radii, the upper bound is
of order $r/(nR)$.
\end{proposition}

For $Q=B_2^n$  the body in Proposition \ref{prop:capsule} is the same euclidean cylinder considered in Laddha-Vempala but our axis-disjoint construction for the upper bound is different. Indeed, there construction gives an upper bound of $r/(Rn^{1/2})$ \cite[Lemma 5]{LaddhaVempala2023} while ours shaves off an additional factor of $n^{1/2}$.

\begin{proposition}[$Q$-cone]
\label{prop:cone}
There are universal constants $C_0,C_1,C>0$ such that the following holds:
Let $n\ge 5$, let $Q\subset\R^n$ be unconditional, let $r>0$ and let $L\ge C_0r$.
The body $K=K_{L,r}^{Q,\rm cone}$ contains the translate
$La_Q/2+(r/8)Q$ and is contained in $(L+r)Q$.
For every $0<s\le e^{-C_1n}$, there exist an integer $k$ and a number $u > 0$ with
 $\ceil{n/3} \le k \le \floor{2n/3}$
such that, upon defining
\[
 N_u(x)=\#\{i:x_i\ge uL(a_Q)_i\},
 \qquad
 A_u=\{x\in K:N_u(x)\le k-1\},
\]
one has
\[
 \frac{\vol(A_u)}{\vol(K)}=s
 \qquad\text{and}\qquad
 \frac{\vol(\partial_0^K A_u)}{\vol(A_u)}
 \le \frac{Cr}{L}.
\]
Thus, in terms of its inner and outer $Q$-radii, the upper bound is of order $r/R$.
\end{proposition}
Consider now the convex body obtained by glueing $K_{L,r}^{Q,cone}$ to $K_{L,r}^Q$ along a common base (using say the bottom base of $K_{L,r}^Q$) while keeping both pieces symmetric about the all 1s vector. The regularity parameters of this convex body are the same (up to constant factors) as those for the cone and cylinder.
Because $K_{L,r}^{Q,cone}$ and $K_{L,r}^Q$ have the same base and height, with the first being a cone and the second a cylinder, the ratio of their volumes is equal to $1/n$. 
In particular one can construct constant-sized axis-disjoint subsets $A,B$ with separator $C$ satisfying $\vol(C)/\min\{\vol(A),\vol(B)\} \le O(r/(Rn))$ by taking the corresponding subsets from Proposition \ref{prop:capsule} and adding $K_{L,r}^{Q,cone}$ to $A$. In addition, the subsets in Proposition~\ref{prop:cone} are such that their $\ell_0$ boundaries do not intersect the base of $K_{L,r}^{Q,cone}$. In particular the subsets from Proposition \ref{prop:cone} are still exponentially small in this body (with possibly a different constant in the exponent) with $\ell_0$ boundaries having relative volume of order $r/R$. These observations immediately imply the following corollary.
\begin{theorem}[Convex body for general upper bound]\label{thm:upper}
There exists $C,c > 0$ such that the following holds: Let $n\ge 5$, let $Q\subset\R^n$ be an unconditional convex body. Let $R,r>0$ where $R\ge Cr$. Then there exists a convex body $K$ that contains a copy of $rQ$ and is contained in a copy of $RQ$. Furthermore $K$ contains a subset $S$ with relative volume being of constant order, but at most $1/2$, with $\ell_0$ boundary having relative volume $O(r/(nR))$. Lastly for any $0 < s < e^{-cn}$ there is a subset $S'$ with relative volume $s$ with $\ell_0$ boundary having relative volume $O(r/R)$.
\end{theorem}

In view of Theorem \ref{thm:upper} we conjecture that Theorem~\ref{thm:main} can be improved by a factor
of $n$, which would be best possible (at least at the extreme set-size ends). We also believe that the intermediate regime for the lower bound would be optimal, but we didn't try to modify the construction used in Theorem \ref{thm:upper} to detect the $\log(e/s)$ factor from Theorem \ref{thm:main}. Assuming the conjectured lowerbound, the
analysis in Remark~\ref{rem:conductance-profile} would imply that the mixing
time of CHAR for a convex body in isotropic position, starting from the
uniform measure on the unit ball, would be 
$O(n^5\log(n)\log(1/\eps))$.

\subsection{Proof idea: Moving labeled volume through its coordinate boundary.}

Our proof approach for Theorem \ref{thm:main} is quite different from those previously used to analyze the $\ell_0$ boundary \cite{LaddhaVempala2023,NarayananRajaramanSrivastava2025,Wadia2024}. 
It instead follows the same boundary-crossing
motivation as the canonical-path and multicommodity-flow methods, where
demand is routed between states and controlled by edge
congestion~\cite{JerrumSinclair1989,DiaconisStroock1991,Sinclair1992}.
For concreteness we mention how the canonical-path approach is used to lower bound the size of the exterior vertex boundary of a finite graph $G = (V,E)$ (see for instance \cite{IsakssonKindlerMossel2012}). For every vertex subset $S$ of $V$ one constructs a collections of paths $(P_\alpha)_{\alpha \in \Lambda}$ with the following properties:
\begin{itemize}
\item Each path starts in $S$ and ends in $S^c$.
\item For every vertex $v \in V$ there are at most $m$ paths which contain $v$.
\end{itemize}
For a given $P_i$ let $(v_1,v_2)$ be the first edge in $P_i$ where $v_1 \in S$ and $v_2 \in S^c$ (the existence of which is guaranteed by the first property). Then $v_2$ lies in the exterior vertex boundary of $S$. Since every path has such a $v_2$ and at most $m$ paths can have the same $v_2$ the size of the exterior vertex boundary is at least $|\Lambda|/m$. Although the approach seems rather clean for finite graphs, it is not immediate as to how one might translate it to our setting. Indeed, the corresponding graph on our convex body would have each point as a vertex and an edge between any two points that agree on all but one coordinate. In particular the graph would have a continuum number of vertices and edges. It is thus unclear as to how one should define paths, how to define overlap, and how to define a lower bound in terms of the number of paths and the amount of overlap (especially if both quantities are infinite). That said, there is a reasonably nice interpration for our setting, which we describe below.

To explain the approach we introduce some nomenclature. Let $K$ be a convex body and let $A \subseteq K$ be Borel. We will refer to $A$ as a source set. A \emph{label} is a source point tracked through every map in a routing,
the common conull Borel set on which the routing maps are defined as its
\emph{label domain}.  A \emph{coordinate routing} is a finite sequence of Borel maps
$\Phi_0,\ldots,\Phi_m$ on a common label domain such that consecutive images
of each label differ by at most 1 coordinate; for a fixed label $x$, the sequence
$\Phi_0(x),\ldots,\Phi_m(x)$ is its \emph{coordinate staircase}.  The routing
has \emph{compression at most $\gamma^{-1}$} if
$\vol(\Phi_j(G))\ge\gamma\vol(G)$ for every time $j$ and every Borel set of
labels $G$.

We start with the following observation:
Suppose that a bounded-compression coordinate routing of length $m$ transports a labeled set $A$ into a new position $A'$. Let $B \subseteq A$ be the subset of labels which are routed to $A'\setminus A$.
Every label in $B$ has a first-exit time where it cross from $A$ to $\partial_0^K A$. Denote the subset of $B$ with first-exit time $i$ by $B_i$.
Since the routing has bounded compression, the volume of the portion of $\partial_0^K A$ occupied by labels from $B_i$ is at least proportional to the volume of $B_i$. Therefore the volume of $\partial_0^K A$ is at least proportional to the maximum volume of the $B_i$. In particular
\[
 \vol(\partial_0^K A)\gtrsim \max_i \vol(B_i) \ge \frac{1}{m}\sum_i \vol(B_i) = \frac{\vol(B)}{m}
\]
We call this argument \emph{first-exit charging}.
Lemma~\ref{lem:first-exit} is the measure-theoretic form of this statement. From the graph perspective, the maps $\Phi_1,\cdots \Phi_m$ collectively encode the canonical paths from $S$ to $S^c$. The volume of $A' \setminus A$ corresponds to the number of canonical paths and $m$ is analogous to the maximum number of overlaps over a single vertex.

Our stategy for constructing a coordinate routing is based on the following toy example: Suppose we want to route some point $x$ to some point $y$, where $K$ is contained in a copy of $RQ$ and both points are contained in $K_\delta^Q:=\{z:z+\delta Q\subset K\}$. We first identify a sequence of interpolant points $(X_{0}, X_{t_1}, \cdots ,X_{t_m})$ along the segment
$\{X_t=(1-t)x+ty, ~0 \le t \le 1\}$. Since $K$ is convex the interpolant points are in $K_\delta^Q$.  For each pair of consecutive interpolant points $X_{t_i},X_{t_{i+1}}$ we can construct a sequence of intermediate points between $X_{t_i}$ to $X_{t_{i+1}}$ by replacing each coordinate entry of $X_{t_i}$ with the corresponding entry in $X_{t_{i+1}}$, one at a time.
We call this block of at most $n$ coordinate moves a \emph{coordinate
sweep}.  Concatenating the sweeps produces a coordinate staircase from $x$ to $y$.
If $\max_i |t_{i+1}-t_i| \le h$ then
every intermediate point of the staircase lies within $O(Rh)$ of the
segment in $Q$-norm, so $h\lesssim\delta/R$ keeps $X_{t_{i+1}}$ inside $X_{t_i} + \delta Q$ for every $i$ and the coordinate staircase inside $K$.
Thus the coordinate route arising from the coordinate staircase uses $O(nR/\delta)$ coordinate moves and the route transports $x$ to $y$.
\begin{remark}[unconditional $Q$ regularity]
In the above example we implicitly assume that intermediate points between $X_{t_i}$ and $X_{t_{i+1}}$ stay within $\delta$ in $Q$-norm of $X_{t_i}$. This is true when $Q$ is unconditional because the coordinate projection of any point contained in an unconditional convex body stays in the body. For general $Q$ (e.g. a slanted cylinder) this need not hold.  An intermediate point need not be contained in $X_{t_i}+\delta Q,X_{t_{i+1}}+\delta Q$ or even in $K$.
\end{remark}

There are two clear obstacles in trying to construct a general routing from the toy example.
First, a general route must route not one point, but an entire set of positive volume. The route must choose destinations for a positive volume of source points and route all those labels simultaneously while having bounded compression. Second, $x$ and $y$ can be arbitrarily close to the boundary. The closer they are to the boundary the smaller $\delta$ has to be made to keep the staircase inside $K$, causing the routing length to explode.

We now explain how we overcome said obstacles.
Instead of trying to construct a general routing that works for any subset $S$, we construct two types of routings. The first routing deals with the boundary obstacle. We decompose
$K$ into dyadic homothetic shells about the center of the inner
$Q$-copy $rQ$.  On each shell, the radial map that doubles homothetic
depth is realized by a coordinate routing.  If many labels leave $S$, the
first-exit lemma charges them to $\partial_0^K S$; if few leave, the
volume expansion of the depth-doubled image forces the occupancy of $S$ to grow geometrically in
the next shell. Consequently, we deduce that either $\vol(\partial_0^K S)$ is large or a
fixed fraction of the relevant volume of $S$ lies in $K_\delta^Q$, with $\delta\asymp r/n$. Since $\vol(S) \le \vol(S^c)$ we can also guarantee that a fixed fraction of the volume of $S^c$ also lies in $K_\delta^Q$ by further scaling down $\delta$ by an absolute constant.

The simultaneous coordinate-routing problem in $K_\delta^Q$ is the crux of the matter. To pick the source-target pairs we essentially do as follows: Choose
equal-volume sets $E\subset S\cap K_\delta^Q$ and
$F\subset (K\setminus S)\cap K_\delta^Q$, and take $\mu = \frac{\ind_{E}}{\vol(E)}$ and $\nu = \frac{\ind_{F}}{\vol(F)}$. Let
$T=\bren:\R^n\to \R^n$ be the Brenier map that pushes $\mu$ forward to $\nu$ (i.e. $T_{\#}\mu = \nu$).  Its displacement interpolation
$(1-t)x+tT(x)$ assigns a transport ray to almost every source label in $E$. For each transport ray we take a sequence of interpolant points (with the choice of $(t_i)_{0 \le i \le m}$ the same for all transport rays) and use coordinate sweeps to construct the corresponding coordinate staircases from $x$ to $T(x)$.
Monotonicity of $T$ gives injectivity of the staircases (i.e. the $i$th step of each staircase is disjoint from one another).  
The Monge--Amp\`ere equation and stability of determinants along the staircases imply bounded compression when the number of points is sufficiently large. A first-exit charging lemma applied to the coordinate-routing then gives the balanced-set lower bound.

For the small-set case, we carry out almost the same argument but with a small change. Transport a source of volume $a\asymp s$ to a
target of fixed positive volume.  The unequal-volume Brenier map has
Jacobian ratio $m\asymp1/s$, and the displacement interpolation at time
$t$ expands volume by at least $m^t$. This forces a large fraction of labels to leave $S$ prematurely. We can thus truncate the routing at the step where the volume of the image of $S$ under the last map, relative to the volume of $S$, is sufficiently large. This occurs at time $t\asymp1/\log(e/s)$, shortens the length of the coordinate
routing by the reciprocal factor, and yields the logarithmic gain in
\eqref{eq:set-main}.  The gain saturates when the shell alternative already
gives $cr/(nR)$.

We remark here that Brenier maps have long been used to prove isoperimetric
inequalities from geometry and analysis (see for instance \cite{Brenier1991,FigalliMaggiPratelli2010,Villani2003}).  Our use of optimal transport extends this trend to the study of $\ell_0$ isoperimetry.

To conclude this section we outline the rest of the paper.
In section~\ref{sec:prelim} we introduce the isoperimetric coefficients, the first-exit lemma,
and the optimal-transport facts used in the proof. In section~\ref{sec:lower} we
prove the lower bound, first by routing through homothetic shells and then
by applying Brenier transport from the homothetic core. In
section~\ref{sec:upper} we give the cylinder and cone constructions and show that they have the desired $\ell_0$ boundary properties. In section~\ref{sec:remarks} we conclude with a brief discussion about two follow-up questions about $\ell_0$ isoperimetry.

\paragraph{Acknowledgment of AI assistance.}
During the preparation of this work, the author used GPT Pro 5.5 and GPT Pro
5.6 Sol to obtain feedback on the development of some of the ideas and to
assist with some initial drafting and the creation of tikz figures. The author reviewed and edited all
AI-assisted material and takes full responsibility for the content of the
paper, including any errors or omissions.

\section{Notation and preliminaries}\label{sec:prelim}
Throughout the paper we assume that $n\ge2$ and that all sets which we consider are Borel unless stated otherwise. We will always use $C,c > 0$ to denote absolute constants, whose value may change between different appearances (even from line to line or across an inequality). 
In Sections \ref{sec:prelim} and \ref{sec:lower}, we assume that every ambient convex body $K$ is normalized to have volume 1.

We write $[n]=\{1,\ldots,n\}$.  For
Borel sets $S\subset K$ and $i\in[n]$, the incidence set
\[
 \mathcal I_i(S,K)
 =\{(x,y)\in S\times(K\setminus S):
       x_j=y_j\text{ for every }j\ne i\}
\]
is Borel in $\R^n\times\R^n$, and
$\partial_{0,i}^K S=\pi_y(\mathcal I_i(S,K))$.  Thus each $\partial_{0,i}^K S$, and
hence $\partial_0^K S$, is analytic and Lebesgue measurable.  We use
$\vol$ to denote $d$-dimensional Lebesgue measure where $d$ is taken to be the ambient dimension of the set which we are measuring. Consequently the dimension can be different across different appearances of $\vol$. For $p \ge 1$ we write $B_p(x,a)$ to denote the $\ell_p$ ball of radius $a$ centered at $x$.

\begin{definition}[Unconditional  body]\label{def:unconditional}
A convex body $Q\subset\R^n$ is \emph{unconditional} (with respect to
the standard coordinate basis) if
\[
 (q_1,\ldots,q_n)\in Q
 \quad\Longrightarrow\quad
 (\varepsilon_1q_1,\ldots,\varepsilon_nq_n)\in Q
 \quad(\varepsilon_i\in\{-1,1\}).
\]
Such a body is centrally symmetric about the origin, so its Minkowski
functional is a norm.  We call it the $Q$-norm and write
\[
 \normQ{z}=\inf\{a>0:z\in aQ\}.
\]
We say that $K$ has \emph{$Q$-regularity with parameters $(r,R)$} if it
satisfies \eqref{eq:unconditional-regularity} for some centers $x_0,y_0$.
We call $x_0+rQ$ and $y_0+RQ$ the \emph{inner} and \emph{outer
$Q$-copies}, respectively.  When $Q=B_p^n$, we also call this
\emph{$\ell_p$-regularity}.
\end{definition}

\begin{lemma}[Coordinate contractions of unconditional bodies]
\label{lem:unconditional-contractions}
Let $Q$ be unconditional.  If
$\Lambda=\operatorname{diag}(\theta_1,\ldots,\theta_n)$ with
$|\theta_i|\le1$, then $\Lambda Q\subset Q$.  In particular, every coordinate
projection $P_J$, $J\subset[n]$, is a
contraction with respect to $\normQ{\cdot}$.
\end{lemma}

\begin{proof}
For $q\in Q$, choose independent random signs $\varepsilon_i\in\{-1,1\}$ with
$\mathbb E\varepsilon_i=\theta_i$.  The random vector 
$(\varepsilon_1q_1,\ldots,\varepsilon_nq_n)$ belongs to $Q$ by
unconditionality. Its expectation is $\Lambda q$, so convexity gives $\Lambda q\in Q$.
Taking $\theta_i=\mathbf 1_{\{i\in J\}}$ proves the projection statement.
\end{proof}

\subsection{$\ell_0$ boundary and coefficients}

\begin{definition}[$\ell_0$ distance and axis-disjointness]\label{def:l0}
For $x\in\R^n$, let
\[
 \|x\|_0=\#\{i\in[n]:x_i\ne0\}.
\]
Two sets $A,B\subset\R^n$ are \emph{axis-disjoint} if
$ \|x-y\|_0\ge2 \text{ for every }x\in A,\ y\in B $.
A \emph{coordinate move} is a pair $x,y$ with $\|x-y\|_0\le1$;
equivalently, the two points lie on a common coordinate fiber.
\end{definition}

\begin{definition}[$\ell_0$ boundary]\label{def:boundary}
For $S\subset K$, define
\begin{align*}
 \partial_{0,i}^K S
 &\defeq \{y\in K\setminus S:\exists x\in S\text{ such that }x_j=y_j\text{ for all }j\ne i\},\\
 \partial_0^K S&\defeq\bigcup_{i=1}^n \partial_{0,i}^K S.
\end{align*}
We omit the superscript $K$ when the ambient body is understood from the context.
\end{definition}

\begin{definition}[One-set and separator coefficients]\label{def:coefficients}
Set
\begin{equation}\label{eq:xi0-def}
 \xi_0(K)\defeq
 \inf_{0<\vol(S)\le 1/2}
 \frac{\vol(\partial_0^K S)}{\vol(S)},
\end{equation}
and
\begin{equation}\label{eq:psi-def}
 \psi_0(K)\defeq
 \inf_{\substack{A,B\subset K\\A,B\text{ axis-disjoint}\\
                    \min\{\vol(A),\vol(B)\}>0}}
 \frac{\vol(K\setminus(A\cup B))}
      {\min\{\vol(A),\vol(B)\}}.
\end{equation}
Here the infima are taken over Borel sets.  We call $\xi_0(K)$ the
\emph{one-set coefficient} and $\psi_0(K)$ the
\emph{separator coefficient}; the latter is the $\ell_0$-isoperimetric
form coefficient introduced in~\cite{Fernandez2026}.
\end{definition}

\begin{proposition}[Equivalence of the formulations]\label{prop:equivalence}
For every $n\ge2$ and every volume-one convex body $K\subset\R^n$,
\begin{equation}\label{eq:coefficient-comparison}
 \frac{\psi_0(K)}{1+\psi_0(K)}\le \xi_0(K)\le\psi_0(K).
\end{equation}
Furthermore, if $\xi_0(K)>0$, then
\begin{equation}\label{eq:coefficient-equivalence}
 \left|\frac{\psi_0(K)}{\xi_0(K)}-1\right|=O(n^{-1/2}).
\end{equation}
\end{proposition}

\begin{proof}
Let $A,B$ be axis-disjoint and, without loss of generality, assume
$\vol(A)\le\vol(B)$.  No point of $B$ belongs to
$\partial_0^K A$, so
\[
 \partial_0^K A\subset K\setminus(A\cup B).
\]
Since $\vol(A)\le1/2$, this gives $\xi_0(K)\le\psi_0(K)$ after taking
the infimum over all admissible $A$ and $B$.

Conversely, fix $A\subset K$ with $0<\vol(A)\le1/2$, and set
\[
 C=\partial_0^K A,\qquad B_*=K\setminus(A\cup C),\qquad
 a=\vol(A),\quad b=\vol(B_*),\quad c=\vol(C),\quad q=\frac ca.
\]
The set $B_*$ is Lebesgue measurable, so it contains a Borel subset
$B\subset B_*$ with $\vol(B)=b$.  The sets $A$ and $B$ are
axis-disjoint, and
\[
 \vol(K\setminus(A\cup B))=c.
\]
If $b\ge a$, then
$\psi_0(K)\le c/a=q$.  If $b<a$, the normalization gives
\[
 b=1-a-c\ge a-c=a(1-q).
\]
Therefore when $q\ge1$ or $b \ge a$ the desired inequality
$q\ge\psi_0(K)/(1+\psi_0(K))$ is automatic. On the other hand, when $q<1$ and $b < a$ the preceding
display yields
\[
 \psi_0(K)\le\frac cb\le\frac{q}{1-q} \implies 
 q\ge\frac{\psi_0(K)}{1+\psi_0(K)}.
\]
Taking the infimum over $A$ proves the first inequality in
\eqref{eq:coefficient-comparison}.

Finally, \eqref{eq:coefficient-comparison} implies
\[
 0\le\frac{\psi_0(K)}{\xi_0(K)}-1\le\psi_0(K).
\]
The universal bound $\psi_0(K)=O(n^{-1/2})$ from
\cite[Theorem~1.4]{Fernandez2026} proves
\eqref{eq:coefficient-equivalence}.
\end{proof}

\subsection{The first-exit lemma}

The following lemma is the first-exit-charging mechanism used throughout Section 3.
It is analogous to the basic boundary-crossing estimate underlying the
canonical-path and multicommodity-flow methods~\cite{JerrumSinclair1989,
DiaconisStroock1991,Sinclair1992}, but it applies to a chosen pairing of
source and target labels.

\begin{lemma}[First-exit lemma]\label{lem:first-exit}
Let $A\subset S\subset K$ be Borel, and let $m\ge1$.  Suppose
\[
 \Phi_0,\Phi_1,\ldots,\Phi_m:A\to K
\]
are Borel maps such that:
\begin{enumerate}[label=(\roman*)]
\item $\Phi_0(x)=x$ and $\Phi_m(x)\in K\setminus S$ for every
      $x\in A$;
\item for every $j=1,\ldots,m$ and $x\in A$, the points
      $\Phi_{j-1}(x)$ and $\Phi_j(x)$ differ in at most one coordinate;
\item there exists $\gamma>0$ such that, for every $j=1,\ldots,m$ and
      every Borel $G\subset A$,
\begin{equation}\label{eq:bounded-compression}
 \vol(\Phi_j(G))\ge\gamma\vol(G).
\end{equation}
\end{enumerate}

Then
\begin{equation}\label{eq:first-exit-bound}
 \vol(\partial_0^K S)\ge \frac{\gamma}{m}\vol(A).
\end{equation}
\end{lemma}

\begin{proof}
For $x\in A$, define its first-exit time by
\[
 \tau(x)=\min\{j\in\{1,\ldots,m\}:\Phi_j(x)\notin S\}.
\]
For $j=1,\ldots,m$, set
\[
 A_j=\{x\in A:\tau(x)=j\}
 =\left(\bigcap_{\ell=0}^{j-1}\Phi_\ell^{-1}(S)\right)
   \cap\Phi_j^{-1}(K\setminus S).
\]
Here $\Phi_\ell^{-1}(\cdot)$ denotes
$\{x\in A:\Phi_\ell(x)\in \cdot\}$.  Since the maps and $S$ are
Borel, the sets $A_1,\ldots,A_m$ are Borel and form a partition of $A$.

If $x\in A_j$, then $\Phi_{j-1}(x)\in S$ and
$\Phi_j(x)\in K\setminus S$.  These two points differ in at most one
coordinate, so $\Phi_j(x)\in\partial_0^K S$.  Thus
$\Phi_j(A_j)\subset\partial_0^K S$.  The image $\Phi_j(A_j)$ is analytic
and hence Lebesgue measurable, and hypothesis~(c), applied to $G=A_j$,
gives
\[
 \gamma\vol(A_j)
 \le \vol(\Phi_j(A_j))
 \le \vol(\partial_0^K S).
\]
Summing over $j$ and using
$\sum_{j=1}^m\vol(A_j)=\vol(A)$ yields
\[
 \gamma\vol(A)\le m\vol(\partial_0^K S),
\]
which is equivalent to \eqref{eq:first-exit-bound}.
\end{proof}

\subsection{\texorpdfstring{$Q$}{Q}-inner parallel bodies and homothetic depth}

Fix a centrally symmetric body $Q$.  In this subsection we assume
$0\in\operatorname{int}K$; in the lower-bound proof this is achieved by
translating the center of the inner $Q$-copy to the origin. For
$\delta\ge0$, define
the \emph{$Q$-inner parallel body}
\begin{equation}\label{eq:inner-parallel}
 K_\delta^Q=\{x\in K:x+\delta Q\subset K\}.
\end{equation}
Thus $x$ has \emph{$Q$-clearance} at least $\delta$ precisely when
$x\in K_\delta^Q$. Convexity of $K_\delta^Q$ immediately follows from convexity of K and \eqref{eq:inner-parallel}. We refer to the complementary set $K\setminus K_\delta^Q$ as the
\emph{boundary layer}. 
The \emph{Minkowski functional} of $K$ is
\[
 \rho_K(x)=\inf\{\lambda\ge0:x\in\lambda K\},
\]
and the corresponding \emph{homothetic depth} is
\begin{equation}\label{eq:depth}
 \eta(x)=1-\rho_K(x),\qquad x\in K.
\end{equation}
Because $Q$ is centrally symmetric it induces a norm $\|\cdot\|_Q := \rho_Q(\cdot)$. We will use the two notations interchangeably throughout the paper.
For $0\le u\le1$, define the \emph{homothetic level set}
\begin{equation}\label{eq:Hu}
 H_u=(1-u)K=\{x\in K:\eta(x)\ge u\}.
\end{equation}

\begin{lemma}[Homothetic depth and $Q$-clearance]\label{lem:depth-clearance}
If $rQ\subset K$, then
\begin{equation}\label{eq:Hu-Kur}
 H_u\subset K_{ur}^Q\qquad(0\le u\le1),
\end{equation}
and
\begin{equation}\label{eq:minkowski-Lipschitz}
 |\rho_K(x)-\rho_K(y)|\le \frac{\normQ{x-y}}{r}
 \qquad(x,y\in\R^n).
\end{equation}
\end{lemma}

\begin{proof}
The case $u=0$ in \eqref{eq:Hu-Kur} is immediate.  If $u>0$,
$x=(1-u)z$ with $z\in K$, and $w\in urQ$, then $w/u\in rQ\subset K$ and
\[
 x+w=(1-u)z+u(w/u)\in K
\]
by convexity.  This proves \eqref{eq:Hu-Kur}.  For
\eqref{eq:minkowski-Lipschitz}, subadditivity of the Minkowski functional
gives
\[
 -\rho_K(y-x)\le\rho_K(x)-\rho_K(y)\le\rho_K(x-y).
\]
Since $rQ\subset K$ and $Q$ is centrally symmetric,
\[
 \rho_K(v)\le\rho_{rQ}(v)=\frac{\normQ{v}}r
 \qquad(v\in\R^n).
\]
Applying this to $v=x-y$ and $v=y-x$ proves the claim.
\end{proof}

A useful consequence of \eqref{eq:Hu-Kur} is the boundary-layer estimate
\begin{equation}\label{eq:boundary-layer}
 \vol(K\setminus K_\delta^Q) \le \vol(K\setminus H_{\delta/r})
 = 1-\left(1-\frac{\delta}{r}\right)^n
 \le \frac{n\delta}{r},
 \qquad 0\le\delta\le r.
\end{equation}
In addition \eqref{eq:minkowski-Lipschitz} implies that $\rho_K$ is differentiable almost everywhere and, at all differentiability points $x$, $|\ip{\nabla \rho_K(x)}{v}| \le \frac{\|v\|_Q}{r}$ for all $v \in \R^n$. Thus for such $x$ we may view $\nabla \rho_K(x)$ as a linear functional on $(\R^n,\|\cdot\|_Q)$ and conclude that 
\begin{equation}\label{eq:dualQ-norm-rho}
\|\nabla \rho_K(x)\|_{Q^*} = \sup_{\|y\|_Q = 1} \ip{\nabla \rho(x)}{y}\le \frac{1}{r},
\end{equation}
where $\|\cdot \|_{Q^*}$ is associated dual-norm to $(\R^n,\|\cdot\|_Q)$.
\subsection{Convex transport and image-volume formulas}

Here we record some facts from convex analysis, optimal transport, and measure theory that we will need later.
For a Borel map $T$, the notation $T_\#\mu=\nu$ means that
$\nu(A)=\mu(T^{-1}(A))$ for every Borel set $A$. Convex functions in
this subsection take values in $\R\cup\{+\infty\}$; such a function is
\emph{proper} if it is not identically $+\infty$. For a proper convex
function $\varphi$ we write dom $\varphi$ to denote $\varphi^{-1}(\R)$. Its Legendre--Fenchel conjugate is
\[
 \varphi^*(y)=\sup_{x\in\R^n}
 \bigl\{\ip{x}{y}-\varphi(x)\bigr\}.
\]
At a point $x$ where a convex function $\varphi$ is Aleksandrov twice
differentiable, we write $\nabla^2\varphi(x)$ for the unique symmetric matrix
satisfying
\[
 \varphi(x+h)=\varphi(x)+\ip{\bren(x)}{h}
 +\frac12 h^{\mathsf T}\nabla^2\varphi(x)h+o\bigl(\norm{h}^2\bigr)
 \qquad (h\to0).
\]
Thus $\nabla^2\varphi(x)$ denotes the Aleksandrov Hessian; it agrees with the
classical Hessian whenever the latter exists.

To implement our mass transport argument we will require the existence of a certain measure-preserving map.
This map is typically referred to as the ``Brenier map", whose existence for certain measures was proven by Brenier in \cite{Brenier1991}. 
A more general statement of existence was proven by McCann in \cite{McCann1995}, which we give below.

\begin{theorem}[Brenier--McCann]\label{thm:brenier}
Let $\mu$ and $\nu$ be arbitrary Borel probability measures on
$\R^n$, and assume that $\mu$ is absolutely continuous. Then there exists a
proper convex function $\varphi$ on $\R^n$ such that $\varphi$ is differentiable $\mu$.a.e and a Borel representative
$T=\bren$ such that
$T_\#\mu=\nu$.  If $\nu$ is also absolutely
continuous, then, after deleting null Borel sets, $T$ is a Borel
bijection whose inverse is $\bren^*$.
\end{theorem}
Because $T$ is the gradient of a convex function it satisfies monotonicity:
\begin{equation}\label{eq:monotonicity}
 \ip{T(x)-T(y)}{x-y}\ge0.
\end{equation}
We now specialize to our setting.
Suppose that $\mu=a^{-1}\mathbf1_E\,dx$ and
$\nu=f^{-1}\mathbf1_F\,dx$, where $E,F$ are bounded Borel sets of
positive volume $a,f$. By convexity $H(x) = \nabla^2 \varphi(x)$ is positive semi-definite at every point where the Aleksandrov Hessian of $\varphi$ exists. Since the Brenier map satisfies the Monge--Amp\`ere equation between $\mu$ and $\nu$ (see \cite[Remark~4.5]{McCann1997}) 
we have
\begin{equation}\label{eq:Hessian-facts}
 \det H(x)=\frac{f}{a}\qquad \text{for } \mu.a.e ~x\in E.
\end{equation}
In the equal-volume case this becomes
\begin{equation}\label{eq:det-one}
 \det H(x)=1\qquad \text{for } \mu.a.e ~x\in E.
\end{equation}
We now record the relevant image-volume formulas that convert determinant lowerbounds on our transport maps to volume preservation statements.
\begin{lemma}[Lipschitz area formula]
\label{lem:lipschitz-area-formula}
Let $E\subset\R^n$, let $T:E\to\R^n$ be Lipschitz, and let
$G\subset E$ be Lebesgue measurable. Then
\begin{equation}\label{eq:lipschitz-area-formula}
 \int_G |\det DT(x)|\,dx
 =\int_{\R^n}N(T,G,y)\,dy,
\end{equation}
where $N(T,G,y)=\#(G\cap T^{-1}(y))$.  In particular, if $T$ is
injective on $G$ then
\[
\int_G |\det DT(x)|\,dx=  \vol(T(G)).
\]
This is a special case of the area formula for Lipschitz maps
(see \cite[Theorem 3.2.3]{Federer1969}) .
\end{lemma}

\begin{lemma}[Monotone image-volume inequality]
\label{lem:monotone-image-volume}
Let $\psi$ be a proper convex function and let
$\Omega=\operatorname{int}(\operatorname{dom}\psi)$.  Define
\[
 J_\psi(x)=\det \nabla^2\psi(x)
\]
at Aleksandrov twice-differentiability points, and set $J_\psi=0$
elsewhere.  Let $M\subset\Omega$ consist of the points at which
$\nabla^2\psi$ exists and is invertible and which are Lebesgue points of
$J_\psi$.  Then every Borel set $G\subset M$ satisfies
\begin{equation}\label{eq:monotone-volume-inequality}
 \vol(\nabla\psi(G))
 \ge \int_G \det \nabla^2\psi(x)\,dx.
\end{equation}
If $\nabla\psi$ is injective on $M$, then equality holds.
\end{lemma}

\begin{proof}
By Corollary~4.3 of McCann~\cite{McCann1997}, $J_\psi$ is locally
integrable and
\[
 (\nabla\psi)_\#(J_\psi\,dx)
 =\mathcal L^n\!\restriction_{\nabla\psi(M)}.
\]
For $B=\nabla\psi(G)\subset\nabla\psi(M)$, it follows that
\[
 \vol(B)
 =\int_{(\nabla\psi)^{-1}(B)}J_\psi(x)\,dx
 \ge\int_GJ_\psi(x)\,dx.
\]
If $\nabla\psi$ is injective on $M$, then
$(\nabla\psi)^{-1}(B)\cap M=G$, giving equality.
\end{proof}
\section{The lower bound}\label{sec:lower}

Fix an unconditional $Q$ and a convex body $K$ satisfying
\eqref{eq:unconditional-regularity}.  Fix $S\subset K$ with
$0<\vol(S)\le1/2$.
Translate $K$ and $S$ by $-x_0$, so the center of the inner $Q$-copy is
the origin, and relabel the translated outer center as $y_0$.
Since $K\subset y_0+RQ$, one has $K-K\subset2RQ$.  In particular, because
$0\in K$, every $x\in K$ satisfies
\begin{equation}\label{eq:diam-bound}
 \normQ{x}\le2R,
\end{equation}
and any $x,z\in K$ satisfy $\normQ{x-z}\le2R$.

\subsection{Shell transport to the homothetic core}\label{sec:shell}

For $0<u<1/2$, define the \emph{homothetic shell}
\begin{equation}\label{eq:shell}
 A_u=H_u\setminus H_{2u}=\{x\in K:u\le\eta(x)<2u\}.
\end{equation}
Fix
\begin{equation}\label{eq:u-star}
 u_*=\frac1{16n},\qquad H_*=H_{u_*},
\end{equation}
and call $H_*$ the \emph{homothetic core}.

\begin{figure}[H]
\centering
\begin{tikzpicture}[
  scale=1.02,
  every node/.style={font=\small},
  ray/.style={-{Latex[length=2.1mm]},teal!72!black,line width=0.85pt},
  stair/.style={orange!88!black,densely dashed,line width=0.75pt,
                rounded corners=0.8pt},
  boundary/.style={draw=black!62,line width=0.55pt},
  omitted/.style={draw=black!42,densely dashed,line width=0.45pt},
  point/.style={circle,inner sep=1.35pt},
  interpoint/.style={circle,inner sep=0.90pt}
]
  % Overall view: the dyadic shell hierarchy approaching the fixed core.
  \fill[gray!12]  (0,1.15) ellipse [x radius=3.20,y radius=1.55];
  \fill[blue!5]   (0,1.15) ellipse [x radius=2.95,y radius=1.43];
  \fill[blue!14]  (0,1.15) ellipse [x radius=2.45,y radius=1.19];
  \fill[blue!22]  (0,1.15) ellipse [x radius=1.90,y radius=0.92];
  \fill[blue!30]  (0,1.15) ellipse [x radius=1.43,y radius=0.70];
  \fill[green!17] (0,1.15) ellipse [x radius=0.72,y radius=0.35];

  \draw[boundary,line width=0.95pt]
    (0,1.15) ellipse [x radius=3.20,y radius=1.55];
  \draw[boundary] (0,1.15) ellipse [x radius=2.95,y radius=1.43];
  \draw[boundary] (0,1.15) ellipse [x radius=2.45,y radius=1.19];
  \draw[boundary] (0,1.15) ellipse [x radius=1.90,y radius=0.92];
  \draw[boundary] (0,1.15) ellipse [x radius=1.43,y radius=0.70];
  \draw[boundary,line width=0.8pt]
    (0,1.15) ellipse [x radius=0.72,y radius=0.35];

  % Dashed contours indicate the unpictured dyadic levels on both sides.
  \draw[omitted] (0,1.15) ellipse [x radius=3.13,y radius=1.52];
  \draw[omitted] (0,1.15) ellipse [x radius=3.06,y radius=1.48];
  \draw[omitted] (0,1.15) ellipse [x radius=1.23,y radius=0.60];
  \draw[omitted] (0,1.15) ellipse [x radius=1.04,y radius=0.51];
  \draw[omitted] (0,1.15) ellipse [x radius=0.86,y radius=0.42];

  % Magnified radial slice of the single transition A_u -> A_{2u}.
  \path[fill=blue!22] (-3.20,-2.43) rectangle (-2.55,-1.03);
  \path[fill=blue!14] (-2.55,-2.43) rectangle (0.15,-1.03);
  \path[fill=blue!5]  (0.15,-2.43) rectangle (2.70,-1.03);
  \path[fill=gray!12] (2.70,-2.43) rectangle (3.20,-1.03);
  \draw[boundary,line width=0.8pt]
    (-3.20,-2.43) rectangle (3.20,-1.03);
  \draw[boundary] (-2.55,-2.43)--(-2.55,-1.03);
  \draw[boundary] (0.15,-2.43)--(0.15,-1.03);
  \draw[boundary] (2.70,-2.43)--(2.70,-1.03);

  % One source piece and its depth-doubled image in the next shell.
  \path[fill=red!24,draw=red!70!black,line width=0.75pt]
    (1.45,-1.73) ellipse [x radius=0.30,y radius=0.24];
  \path[fill=blue!42,draw=blue!82!black,line width=0.75pt]
    (-1.05,-1.73) ellipse [x radius=0.40,y radius=0.30];

  \coordinate (source) at (1.45,-1.57);
  \coordinate (target) at (-1.05,-1.93);
  \coordinate (interp) at ($(source)!0.5!(target)$);

  % The transport ray and its coordinate-staircase realization.
  \draw[ray] (source) -- (target);
  \draw[stair]
    (source) -- (interp |- source) -- (interp)
             -- (target |- interp) -- (target);
  \node[interpoint,fill=yellow!75!orange,draw=orange!75!black,thin]
    at (interp) {};
  \node[point,fill=red!78!black] at (source) {};
  \node[point,fill=blue!78!black] at (target) {};

  % Legend, kept entirely to the right of both scales of the diagram.
  \begin{scope}[shift={(3.55,2.70)}]
    \path[draw=black!32,fill=white,rounded corners=1.5pt,line width=0.45pt]
      (0,0) rectangle (5.90,-5.13);
    \node[anchor=west,font=\scriptsize\bfseries] at (0.18,-0.24) {Legend};

    \path[fill=gray!12,draw=black!62,line width=0.55pt]
      (0.18,-0.47) rectangle (0.43,-0.63);
    \node[anchor=west,font=\scriptsize] at (0.55,-0.55)
      {$K\setminus H_u$: omitted shallower shells};

    \path[fill=blue!5,draw=black!45,line width=0.45pt]
      (0.18,-0.86) rectangle (0.43,-1.02);
    \node[anchor=west,font=\scriptsize] at (0.55,-0.94)
      {$A_u=H_u\setminus H_{2u}$: source shell};

    \path[fill=blue!14,draw=black!45,line width=0.45pt]
      (0.18,-1.25) rectangle (0.43,-1.41);
    \node[anchor=west,font=\scriptsize] at (0.55,-1.33)
      {$A_{2u}=H_{2u}\setminus H_{4u}$: image shell};

    \path[fill=blue!22,draw=black!45,line width=0.45pt]
      (0.18,-1.64) rectangle (0.43,-1.80);
    \node[anchor=west,font=\scriptsize] at (0.55,-1.72)
      {$A_{4u}=H_{4u}\setminus H_{8u}$: next shell};

    \path[fill=blue!30,draw=black!45,line width=0.45pt]
      (0.18,-2.03) rectangle (0.43,-2.19);
    \node[anchor=west,font=\scriptsize] at (0.55,-2.11)
      {$H_{8u}\setminus H_*$: further dyadic shells};

    \draw[omitted] (0.18,-2.50)--(0.43,-2.50);
    \node[anchor=west,font=\scriptsize] at (0.55,-2.50)
      {dashed contours: omitted dyadic levels};

    \path[fill=green!17,draw=black!58,line width=0.55pt]
      (0.18,-2.81) rectangle (0.43,-2.97);
    \node[anchor=west,font=\scriptsize] at (0.55,-2.89)
      {$H_*$: homothetic core};

    \path[fill=red!24,draw=red!70!black,line width=0.55pt]
      (0.18,-3.20) rectangle (0.43,-3.36);
    \node[anchor=west,font=\scriptsize] at (0.55,-3.28)
      {\textcolor{red!70!black}{$E\subset A_u$}: source piece};

    \path[fill=blue!42,draw=blue!82!black,line width=0.55pt]
      (0.18,-3.59) rectangle (0.43,-3.75);
    \node[anchor=west,font=\scriptsize] at (0.55,-3.67)
      {\textcolor{blue!82!black}{$T(E)$}: depth-doubled image in $A_{2u}$};

    \node[point,fill=red!78!black] at (0.20,-4.06) {};
    \node[interpoint,fill=yellow!75!orange,draw=orange!75!black,thin]
      at (0.305,-4.06) {};
    \node[point,fill=blue!78!black] at (0.41,-4.06) {};
    \node[anchor=west,font=\scriptsize] at (0.55,-4.06)
      {source, interpolant, and target points};

    \draw[ray] (0.18,-4.43)--(0.43,-4.43);
    \node[anchor=west,font=\scriptsize] at (0.55,-4.43)
      {solid arrow: transport ray};

    \draw[stair]
      (0.18,-4.89)--(0.33,-4.89)--(0.33,-4.74)--(0.45,-4.74);
    \node[anchor=west,font=\scriptsize] at (0.55,-4.82)
      {dashed staircase: coordinate moves};
  \end{scope}
\end{tikzpicture}
\caption{Shell transport toward the homothetic core.  The upper ellipse shows
the dyadic hierarchy $A_u,A_{2u},A_{4u},\ldots$ approaching $H_*$; dashed
contours stand for omitted levels near both $\partial K$ and the core.  The
lower strip illustrates one radial ``slice'' of the transition from $A_u$ to
$A_{2u}$.  The depth-doubling map sends a source piece $E$ to an expanded
image $T(E)$ in $A_{2u}$.  The solid arrow is the transport ray; the
dashed staircase realizes it by coordinate moves through the yellow
interpolant.  Endpoint mass that leaves $S$ is charged to
$\partial_0^K S$; if little mass leaves, volume expansion forces increased
occupancy of the next shell.}
\label{fig:shell-transport}
\end{figure}
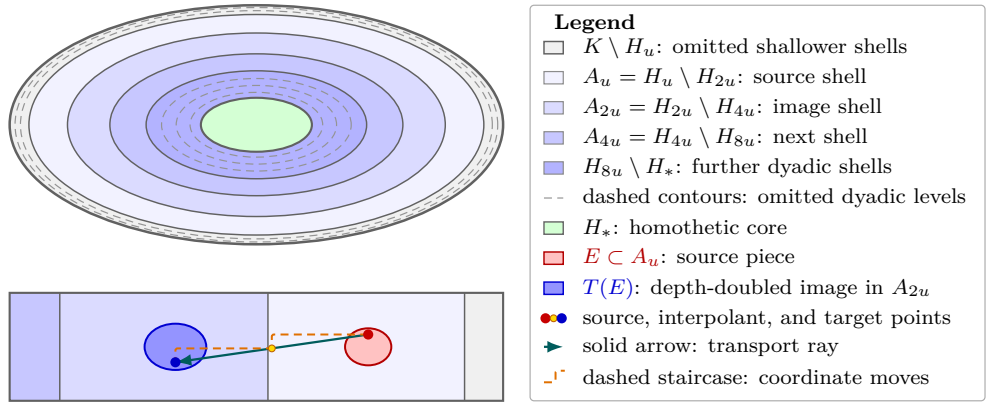

\subsubsection{The shell-to-shell transport map}

For $0 < u < 1/4$ the map $t \mapsto 2t-1$ sends the interval
$(1-2u,1-u]$ to
$(1-4u,1-2u]$.  Accordingly, define the radial
map
\begin{equation}\label{eq:radial-map}
 T(x)=\frac{2\rho_K(x)-1}{\rho_K(x)}x
 \qquad\bigl(\rho_K(x)>1/2\bigr).
\end{equation}
Thus $\eta(T(x)) = 1-\rho_K(T(x))=2(1-\rho_K(x)) = 2\eta(x)$. In particular $T(A_u) = A_{2u}$ for $0 < u < 1/4$.

\begin{lemma}[shell-to-shell map]\label{lem:shell-doubling}
For $0<u<1/4$, the map $T$ is a Borel bijection from $A_u$ to
$A_{2u}$.  If
$u\le1/(16n)$, then every Borel $G\subset A_u$ satisfies
\begin{equation}\label{eq:depth-doubling-volume}
 \frac{3}{2}\vol(G) \le \vol(T(G)) \le 2 \vol(G).
\end{equation}
\end{lemma}

\begin{proof}
The fact that it is a Borel bijection is immediate from \eqref{eq:radial-map}. To prove \eqref{eq:depth-doubling-volume} we will use Lemma \ref{lem:lipschitz-area-formula}. The fact that $T$ is Lipschitz follows from \eqref{eq:minkowski-Lipschitz} and the fact that $\rho_K > 1/2$ on $A_u$. Next, for every differentiability point $x \in A_u$ of $\rho_K$  write $x = tz$ where $t = \rho_K(x)$ and $z$ satisfies $\rho_K(z) = 1$. Note that $\rho_K(sz) = s$ for every $s$. Hence $\frac{d}{ds}\rho_K(sz) = 1$. By the multivariable chain rule 
\begin{equation}\label{eq:minkowski-grad}
\frac{d}{ds}\rho_K(sz) =\sum_i \frac{\partial}{\partial x_i}\rho_K(sz) \cdot \frac{d}{ds}sz_i = \ip{\nabla \rho_K(sz)}{z}.
\end{equation}
In particular taking $s = t$ gives $\ip{\nabla \rho_K(x)}{z} = 1$. Next, using the quotient rule and product rule one has

\[
\left(\frac{\partial T}{\partial x_j}\right)_i = \frac{\partial}{\partial x_j} \left(2 - \frac{1}{\rho(x)}\right)x_i
=  \frac{x_i\frac{\partial}{\partial x_j}\rho(x)}{\rho(x)^2} + \left(2 - \frac{1}{\rho(x)}\right)\cdot \ind_{i = j}
\]

Therefore,
\[
DT(x) = \underbrace{\frac{2\rho(x)-1}{\rho(x)} \cdot I}_{=: A} + \underbrace{\frac{1}{\rho(x)^2} \cdot x (\nabla \rho(x))^\top}_{=: uv^\top}
\]
In particular $DT(x)$ is the sum of an invertible matrix and a rank 1 matrix. Recall now the determinant formula for a rank-1 perturbation of an invertible matrix:
\begin{equation}\label{eq:rank-one-det}
\det(A + uv^\top) = \det(A)(1 + v^\top A^{-1}u).
\end{equation}
Applying \eqref{eq:rank-one-det} to $DT(x)$ and simplifying the resulting expression yields
\begin{equation}\label{eq:DT-ident}
\det DT(x) = 2 \cdot \left(\frac{2\rho_K(x)-1}{\rho_K(x)}\right)^{n-1}
\end{equation}
Since $\rho(x) \in (1-2u,1-u]$ one has 
\begin{equation}\label{eq:DT-lb}
\left(\frac{2\rho_K(x)-1}{\rho_K(x)}\right)^{n-1} \ge \left(\frac{1-4u}{1-u}\right)^{n-1} \ge (1-4u)^n \ge 1-4un \ge \frac{3}{4}.
\end{equation}
Since $\rho$ is differentiable almost everywhere we may apply Lemma \ref{lem:lipschitz-area-formula} to $T$ and, estimating det $DT(x)$ with \eqref{eq:DT-ident} and \eqref{eq:DT-lb}, deduce for every Borel $G$ that
\[
\vol(T(G)) = \int_G |\det DT(x)| \in [(3/2)\vol(G),2\vol(G)],
\]
which confirms \eqref{eq:depth-doubling-volume}.
\end{proof}

\subsubsection{Coordinate staircases and localized first-exit charging}

\begin{lemma}[Shell-to-shell transport]\label{lem:shell-path}
Let $u\le1/(16n)$, let $E\subset A_u$ be Borel, and let $T$ be
defined by radial-map defined in \eqref{eq:radial-map}.  There are Borel maps
\[
 \Psi_0,\ldots,\Psi_M:E\to K
\]
\begin{enumerate}[label=(\roman*)]
\item \label{item:shell-route} $\Psi_0=\Id$, $\Psi_M= T$, and the images of a point under consecutive maps differ in at most one coordinate;
\item \label{item:shell-steps} $M\le CnR/r$.
\item \label{item:shell-injective} every $\Psi_m$ is injective on $E$.
\item \label{itemL:shell-bounded-compression} every $\Psi_m$ satisfies
\begin{equation}\label{eq:shell-bounded-compression}
 \vol(\Psi_m(G))\ge\frac12\vol(G)
 \quad\text{for every Borel }G\subset E.
\end{equation}
\item \label{item:shell-localization} every staircase point lies in the \emph{localized shell band}
\begin{equation}\label{eq:Uu}
 U_u=\{z\in K:u/2\le\eta(z)\le5u\}.
\end{equation}
\end{enumerate}
\end{lemma}
The proof is deferred to the appendix (see \ref{sec:shell-path proof}).

\begin{lemma}[Localized first-exit bound]\label{lem:local-charge}
Let $E\subset S\cap A_u$ be Borel.  Then
\begin{equation}\label{eq:local-exit}
 \vol(T(E)\setminus S)
 \le C\frac{nR}{r}\vol(\partial_0^K S\cap U_u).
\end{equation}
\end{lemma}

\begin{proof}
Let $E'=\{x\in E:T(x)\notin S\}$.  By
Lemma~\ref{lem:shell-doubling},
\[
 \vol(T(E)\setminus S)=\vol(T(E'))\le2\vol(E').
\]
Apply Lemma~\ref{lem:first-exit} to the Borel label set $A=E'$ and
the coordinate staircase from Lemma~\ref{lem:shell-path}.  At each
first-exit time, the
endpoint belongs to $\partial_0^K S\cap U_u$.  With $\gamma=1/2$ and
$M\le CnR/r$, the resulting localized first-exit estimate gives the
claim.
\end{proof}

\subsubsection{The dyadic recurrence}

For an integer $J\ge1$, let $\eps=2^{-J}u_*$, $u_j=2^j\eps$, and
\[
 A_j=A_{u_j},\qquad a_j=\vol(S\cap A_j),\qquad 0\le j\le J.
\]
For $j<J$, apply the shell-to-shell map to $E_j=S\cap A_j$, and let
$e_j=\vol(T(E_j)\setminus S)$.

\begin{lemma}\label{lem:recurrence}
For $0\le j<J$,
\begin{equation}\label{eq:recurrence}
 \frac32a_j\le e_j+a_{j+1}.
\end{equation}
Consequently,
\begin{equation}\label{eq:telescoping}
 \sum_{j=0}^{J-1}a_j
 \le C\left(\sum_{j=0}^{J-1}e_j+a_J\right).
\end{equation}
\end{lemma}

\begin{proof}
The endpoint image lies in $A_{j+1}$ and, by
Lemma~\ref{lem:shell-doubling}, has volume at least $(3/2)a_j$.  Its portion
outside $S$ has volume $e_j$; its portion inside $S$ has volume at most
$a_{j+1}$.  This proves \eqref{eq:recurrence}.  Summing over
$j$ gives
\[
\frac32a_0+\frac12\sum_{j=1}^{J-1}a_j
\le\sum_{j=0}^{J-1}e_j+a_J.
\]
Since the left-hand side is at least
$(1/2)\sum_{j=0}^{J-1}a_j$, this proves
\eqref{eq:telescoping}, with the explicit constant $2$.
\end{proof}

If $U_{u_j}$ and $U_{u_k}$ overlap, then
$2^{|j-k|}<10$, and hence $|j-k|<4$.  Partition
$\{0,\ldots,J-1\}$ into its four residue classes modulo $4$.
The bands indexed by any one class are pairwise disjoint, and some class
$I$ satisfies
\[
 \sum_{j\in I}e_j\ge\frac14\sum_{j=0}^{J-1}e_j.
\]
Applying Lemma~\ref{lem:local-charge} gives
\begin{equation}\label{eq:sum-exits}
 \sum_{j=0}^{J-1}e_j
 \le4\sum_{j\in I}e_j
 \le C\frac{nR}{r}
       \sum_{j\in I}\vol(\partial_0^K S\cap U_{u_j})
 \le C\frac{nR}{r}\vol(\partial_0^K S).
\end{equation}

\begin{proposition}[Shell reduction]\label{prop:shell-reduction}
With $H_*$ defined by \eqref{eq:u-star}, every Borel $S\subset K$ satisfies
\begin{equation}\label{eq:shell-dichotomy}
 \vol(S)
 \le C\left[
   \frac{nR}{r}\vol(\partial_0^K S)
   +\vol(S\cap H_*)
 \right].
\end{equation}
\end{proposition}

\begin{proof}
Up to null boundaries,
\[
 H_\eps=\left(\bigcup_{j=0}^{J-1}A_j\right)\cup H_*.
\]
As $J\to\infty$, $\eps\downarrow0$ and
$H_\eps\uparrow\operatorname{int}K$. Furthermore, since the boundary of a convex body has zero Lebesgue measure, $\eps \downarrow 0$ implies $\vol(S\cap H_\eps)\to\vol(S)$. Therefore combining \eqref{eq:telescoping} and \eqref{eq:sum-exits} and noting that 
$a_J\le\vol(S\cap H_*)$, we deduce \eqref{eq:shell-dichotomy}.
\end{proof}

\subsection{Brenier transport from the homothetic core}\label{sec:deep}
In view of Proposition \ref{prop:shell-reduction} we are left with proving a lower bound on $\partial_0^K(S)$ when the majority of $S$ lies in the homothetic core:
\begin{equation}\label{eq:deep-set}
 D=S\cap H_*,\qquad d=\vol(D).
\end{equation}
By Lemma~\ref{lem:depth-clearance},
\begin{equation}\label{eq:delta-core}
 D\subset K_\delta^Q,
 \qquad
 \delta=u_*r=\frac{r}{16n}.
\end{equation}
We provide two coordinate routing schemes for inside the core. The first scheme is used for routing two sets of equal positive volume. The second scheme is for routing two sets of positive volume where the target set is larger than the source set.
\begin{lemma}[Equal-volume Brenier transport]\label{lem:deep-transport}
Let $E \subset H^*,F\subset K_\delta^Q$ be Borel sets of equal positive volume.
Then there are conull Borel sets $E_0\subset E$, $F_0\subset F$, and Borel maps
\[
 \Phi_0,\ldots,\Phi_M:E_0\to K
\]
with the following properties:
\begin{enumerate}[label=(\roman*)]
\item \label{item:core-equal-route}$\Phi_0=\Id$, $\Phi_M:E_0\to F_0$ is a bijection and the image of a point under consecutive maps differ in at most one coordinate.
\item \label{item:core-equal-steps}
\begin{equation}\label{eq:deep-length}
 M\le Cn\left(n+\frac{R}{\delta}\right).
\end{equation}

\item \label{item:core-equal-injectivity} Every $\Phi_j$ is injective on $E_0$.
\item \label{item:deep-bounded-compression} Every $\Phi_j$ satisfies
\begin{equation}\label{eq:deep-bounded-compression}
 \vol(\Phi_j(G))\ge\frac12\vol(G)
 \qquad\text{for every Borel }G\subset E_0.
\end{equation}
\end{enumerate}
\end{lemma}

The proof is deferred to the appendix (see \ref{sec:deep-transport proof}).

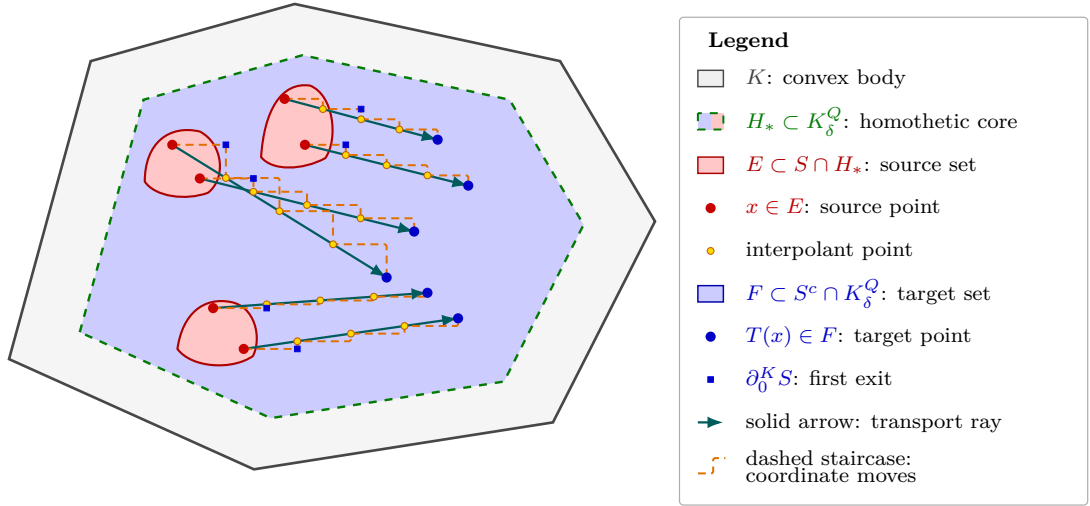
\begin{figure}[H]
\centering
\begin{tikzpicture}[
  scale=1.35,
  every node/.style={font=\small},
  ray/.style={-{Latex[length=2.1mm]},teal!72!black,line width=0.85pt},
  stair/.style={orange!88!black,densely dashed,line width=0.75pt,
                rounded corners=0.8pt},
  point/.style={circle,inner sep=1.35pt},
  interpoint/.style={circle,inner sep=0.90pt},
  exitpoint/.style={rectangle,inner sep=1.05pt,fill=blue!78!black,
                   draw=blue!92!black,line width=0.35pt}
]
  % A deliberately non-product convex polygon.
  \path[fill=gray!8,draw=black!72,line width=1pt]
    (-3.15,-1.20) -- (-2.35,1.72) -- (-0.35,2.28)
    -- (2.25,1.72) -- (3.18,0.15) -- (2.18,-1.82)
    -- (-0.75,-2.28) -- cycle;

  % Homothetic core (1-c_0/n)K.
  \begin{scope}[scale=0.78]
    \path[fill=blue!20,draw=green!48!black,dashed,line width=0.9pt]
      (-3.15,-1.20) -- (-2.35,1.72) -- (-0.35,2.28)
      -- (2.25,1.72) -- (3.18,0.15) -- (2.18,-1.82)
      -- (-0.75,-2.28) -- cycle;
  \end{scope}

  % Three components of the source set E, spread through the homothetic core.
  \path[fill=red!22,draw=red!68!black,line width=0.8pt]
    (-1.82,0.54)
    .. controls (-1.82,0.92) and (-1.58,1.10) .. (-1.28,1.03)
    .. controls (-1.05,0.91) and (-1.02,0.57) .. (-1.20,0.43)
    .. controls (-1.42,0.35) and (-1.72,0.39) .. (-1.82,0.54)
    -- cycle;
  \path[fill=red!22,draw=red!68!black,line width=0.8pt]
    (-0.68,0.78)
    .. controls (-0.65,1.24) and (-0.48,1.52) .. (-0.20,1.48)
    .. controls ( 0.03,1.39) and ( 0.08,0.96) .. (-0.05,0.77)
    .. controls (-0.27,0.65) and (-0.57,0.64) .. (-0.68,0.78)
    -- cycle;
  \path[fill=red!22,draw=red!68!black,line width=0.8pt]
    (-1.50,-1.10)
    .. controls (-1.45,-0.75) and (-1.24,-0.59) .. (-0.98,-0.64)
    .. controls (-0.74,-0.68) and (-0.67,-0.99) .. (-0.76,-1.17)
    .. controls (-0.96,-1.31) and (-1.36,-1.29) .. (-1.50,-1.10)
    -- cycle;

  % Six schematic transport samples with vertically staggered targets.
  \coordinate (x1) at (-1.55, 0.90);
  \coordinate (x2) at (-1.28, 0.57);
  \coordinate (x3) at (-0.45, 1.35);
  \coordinate (x4) at (-0.25, 0.90);
  \coordinate (x5) at (-1.15,-0.70);
  \coordinate (x6) at (-0.85,-1.10);
  \coordinate (y1) at ( 0.55,-0.400);
  \coordinate (y2) at ( 0.82, 0.050);
  \coordinate (y3) at ( 1.05, 0.950);
  \coordinate (y4) at ( 1.35, 0.500);
  \coordinate (y5) at ( 0.95,-0.550);
  \coordinate (y6) at ( 1.25,-0.800);

  % Three displacement-interpolation states on each ray.
  \coordinate (m1a) at (-1.025, 0.575);
  \coordinate (m1b) at (-0.500, 0.250);
  \coordinate (m1c) at ( 0.025,-0.075);
  \coordinate (m2a) at (-0.755, 0.440);
  \coordinate (m2b) at (-0.230, 0.310);
  \coordinate (m2c) at ( 0.295, 0.180);
  \coordinate (m3a) at (-0.075, 1.250);
  \coordinate (m3b) at ( 0.300, 1.150);
  \coordinate (m3c) at ( 0.675, 1.050);
  \coordinate (m4a) at ( 0.150, 0.800);
  \coordinate (m4b) at ( 0.550, 0.700);
  \coordinate (m4c) at ( 0.950, 0.600);
  \coordinate (m5a) at (-0.625,-0.663);
  \coordinate (m5b) at (-0.100,-0.625);
  \coordinate (m5c) at ( 0.425,-0.588);
  \coordinate (m6a) at (-0.325,-1.025);
  \coordinate (m6b) at ( 0.200,-0.950);
  \coordinate (m6c) at ( 0.725,-0.875);

  \draw[ray] (x1) -- (y1);
  \draw[ray] (x2) -- (y2);
  \draw[ray] (x3) -- (y3);
  \draw[ray] (x4) -- (y4);
  \draw[ray] (x5) -- (y5);
  \draw[ray] (x6) -- (y6);

  % One two-move coordinate sweep joins every pair of marked interpolation states.
  \draw[stair]
    (x1) -- (m1a |- x1) -- (m1a)
         -- (m1b |- m1a) -- (m1b)
         -- (m1c |- m1b) -- (m1c)
         -- (y1 |- m1c) -- (y1);
  \draw[stair]
    (x2) -- (m2a |- x2) -- (m2a)
         -- (m2b |- m2a) -- (m2b)
         -- (m2c |- m2b) -- (m2c)
         -- (y2 |- m2c) -- (y2);
  \draw[stair]
    (x3) -- (m3a |- x3) -- (m3a)
         -- (m3b |- m3a) -- (m3b)
         -- (m3c |- m3b) -- (m3c)
         -- (y3 |- m3c) -- (y3);
  \draw[stair]
    (x4) -- (m4a |- x4) -- (m4a)
         -- (m4b |- m4a) -- (m4b)
         -- (m4c |- m4b) -- (m4c)
         -- (y4 |- m4c) -- (y4);
  \draw[stair]
    (x5) -- (m5a |- x5) -- (m5a)
         -- (m5b |- m5a) -- (m5b)
         -- (m5c |- m5b) -- (m5c)
         -- (y5 |- m5c) -- (y5);
  \draw[stair]
    (x6) -- (m6a |- x6) -- (m6a)
         -- (m6b |- m6a) -- (m6b)
         -- (m6c |- m6b) -- (m6c)
         -- (y6 |- m6c) -- (y6);

  \foreach \p in {x1,x2,x3,x4,x5,x6}
    \node[point,fill=red!78!black] at (\p) {};
  \foreach \p in {y1,y2,y3,y4,y5,y6}
    \node[point,fill=blue!78!black] at (\p) {};

  % Mark three intermediate states along every displacement ray.
  \foreach \p in {m1a,m1b,m1c,m2a,m2b,m2c,m3a,m3b,m3c,
                  m4a,m4b,m4c,m5a,m5b,m5c,m6a,m6b,m6c}
    \node[interpoint,fill=yellow!75!orange,draw=orange!75!black,thin]
      at (\p) {};

  % The first staircase corner outside each source component.
  \node[exitpoint] at (m1a |- x1) {};
  \node[exitpoint] at (m2a |- x2) {};
  \node[exitpoint] at (m3b |- m3a) {};
  \node[exitpoint] at (m4a |- x4) {};
  \node[exitpoint] at (m5a |- x5) {};
  \node[exitpoint] at (m6a |- x6) {};

  % Legend, kept entirely to the right of the geometric diagram.
  \begin{scope}[shift={(3.42,2.15)}]
    \path[draw=black!32,fill=white,rounded corners=1.5pt,line width=0.45pt]
      (0,0) rectangle (4.00,-4.78);
    \node[anchor=west,font=\scriptsize\bfseries] at (0.18,-0.24) {Legend};

    \path[fill=gray!12,draw=black!72,line width=0.55pt]
      (0.18,-0.52) rectangle (0.43,-0.68);
    \node[anchor=west,font=\scriptsize] at (0.55,-0.60)
      {\textcolor{black!75}{$K$}: convex body};

    \path[fill=blue!20,draw=none]
      (0.18,-0.94) rectangle (0.305,-1.10);
    \path[fill=red!22,draw=none]
      (0.305,-0.94) rectangle (0.43,-1.10);
    \path[fill=none,draw=green!48!black,dashed,line width=0.75pt]
      (0.18,-0.94) rectangle (0.43,-1.10);
    \node[anchor=west,font=\scriptsize] at (0.55,-1.02)
      {\textcolor{green!48!black}{$H_*\subset K_\delta^Q$}: homothetic core};

    \path[fill=red!22,draw=red!68!black,line width=0.55pt]
      (0.18,-1.37) rectangle (0.43,-1.53);
    \node[anchor=west,font=\scriptsize] at (0.55,-1.45)
      {\textcolor{red!68!black}{$E\subset S\cap H_*$}: source set};

    \node[point,fill=red!78!black] at (0.305,-1.87) {};
    \node[anchor=west,font=\scriptsize] at (0.55,-1.87)
      {\textcolor{red!78!black}{$x\in E$}: source point};

    \node[interpoint,fill=yellow!75!orange,draw=orange!75!black,thin]
      at (0.305,-2.29) {};
    \node[anchor=west,font=\scriptsize] at (0.55,-2.29)
      {interpolant point};

    \path[fill=blue!20,draw=blue!68!black,line width=0.55pt]
      (0.18,-2.63) rectangle (0.43,-2.79);
    \node[anchor=west,font=\scriptsize] at (0.55,-2.71)
      {\textcolor{blue!78!black}{$F\subset S^c\cap K_\delta^Q$}: target set};

    \node[point,fill=blue!78!black] at (0.305,-3.13) {};
    \node[anchor=west,font=\scriptsize] at (0.55,-3.13)
      {\textcolor{blue!78!black}{$T(x)\in F$}: target point};

    \node[exitpoint] at (0.305,-3.55) {};
    \node[anchor=west,font=\scriptsize] at (0.55,-3.55)
      {\textcolor{blue!78!black}{$\partial_0^K S$}: first exit};

    \draw[ray] (0.18,-3.97)--(0.43,-3.97);
    \node[anchor=west,font=\scriptsize] at (0.55,-3.97)
      {solid arrow: transport ray};

    \draw[stair] (0.18,-4.47)--(0.33,-4.47)--(0.33,-4.32)--(0.45,-4.32);
    \node[anchor=west,font=\scriptsize,align=left] at (0.55,-4.40)
      {dashed staircase:\\[-1mm]coordinate moves};
  \end{scope}
\end{tikzpicture}

\caption{Brenier transport from the homothetic core into the $Q$-inner
parallel body.  A source set $E\subset S\cap H_*$, shown in three
components, is transported to $F\subset S^c\cap K_\delta^Q$.  Each source
point is sent to a target point along a transport ray, and the dashed
coordinate staircase realizes that ray by coordinate moves.  The blue squares represent
the first exits in $\partial_0^K S$ charged by the corresponding source labels. A coordinate staircase itself need not be entirely contained in $K_\delta^Q$ but is always contained in $K$.}
\label{fig:brenier-core-transport}
\end{figure}

\begin{lemma}[Unequal-volume Brenier transport]
\label{lem:early-expansion}
Let $E \subset H^*,F\subset K_\delta^Q$ be Borel sets of
volumes $a$ and $f$, where $0<a\le f$, and put $m=f/a$.  For every
$0<\theta\le1/2$, there are conull Borel sets $E_0\subset E$,
$F_0\subset F$, a Borel Brenier bijection $T:E_0\to F_0$, and Borel maps
\[
 \Phi_0,\ldots,\Phi_M:E_0\to K
\]
such that:
\begin{enumerate}[label=(\roman*)]
\item \label{item:core-unequal-route}
$\Phi_0=\Id$ and the image of a point under consecutive maps differ in at one most coordinate.    
\item\label{item:core-unequal-steps}
\begin{equation}\label{eq:variable-time-length}
 M\le Cn\left(1+\frac{nR}{r}\theta\right).
 \end{equation}
\item \label{item:core-unequal-injectivity} Every $\Phi_j$ is injective on $E_0$.
\item \label{item:early-bounded-compression} Every $\Phi_j$ satisfies
\begin{equation}\label{eq:early-bounded-compression}
 \vol(\Phi_j(G))\ge\frac12\vol(G)
 \qquad\text{for every Borel }G\subset E_0.
\end{equation}
\item \label{item:core-unequal-expansion}
\begin{equation}\label{eq:early-expansion}
 \vol(\Phi_M(G))\ge m^\theta\vol(G)
 \qquad\text{for every Borel }G\subset E_0.
\end{equation}
\end{enumerate}
\end{lemma}
\begin{proof}
  The overall proof is essentially the same as the proof of Lemma \ref{lem:deep-transport} (see \ref{sec:deep-transport proof} for details) so we only mention the differences. The maps used here are the same $\Phi_{\ell,k}$ as in the proof of Lemma \ref{lem:deep-transport} except we only take $\ell$ up to $\ceil{\theta L}$. The proofs of \ref{item:core-unequal-route}\ref{item:core-unequal-injectivity} and \ref{item:early-bounded-compression} from Lemma \ref{lem:early-expansion} immediately follow from the proofs of \ref{item:core-equal-route}\ref{item:core-equal-injectivity} and \ref{item:deep-bounded-compression} from Lemma \ref{lem:deep-transport}. For \ref{item:core-unequal-expansion} from Lemma \ref{lem:early-expansion} there is only 1 minor difference. Because the Hessian of $\varphi$ from Theorem \ref{thm:brenier} satisfies \eqref{eq:Hessian-facts} the determinant lower bound proof for \ref{item:core-equal-injectivity} actually implies 
  \[
  \det D\Phi_{\ell+1,0} \ge m^{t_\ell},\qquad \det D\Phi_{\ell,k} \ge \frac{1}{2}m^{t_\ell},\qquad \text{ for all } 0 \leq k \leq n, ~0 \leq \ell \leq \ceil{\theta L}.
  \]
  Since $m \ge 1$ and $t_M \ge \theta$, the proofs of \ref{item:core-unequal-expansion} and \ref{item:core-unequal-steps} from Lemma \ref{lem:early-expansion} immediately follow from the proofs of \ref{item:deep-bounded-compression} and \ref{item:core-equal-steps}.
\end{proof}
We first record a set-size-independent boundary estimate and then use Lemma~\ref{lem:early-expansion} to
obtain the profile.  The
boundary-layer estimate \eqref{eq:boundary-layer} with
$\delta=r/(16n)$ gives
\begin{equation}\label{eq:core-complement}
 \vol(K\setminus K_\delta^Q)\le\frac1{16}.
\end{equation}
Since $\vol(S)\le1/2$,
\begin{equation}\label{eq:enough-complement}
 \vol(K_\delta^Q\setminus S)\ge1-\frac1{16}-\frac12=\frac7{16}.
\end{equation}

\begin{proposition}\label{prop:linear-fallback}
Every Borel $S\subset K$ with $0<\vol(S)\le1/2$ satisfies
\begin{equation}\label{eq:linear-fallback}
 \frac{\vol(\partial_0^K S)}{\vol(S)}\ge \frac{cr}{n^2R}.
\end{equation}
\end{proposition}
\begin{proof}
Recall the parameters in \eqref{eq:deep-set}. If $d>0$ then \eqref{eq:enough-complement} and \eqref{eq:boundary-layer} imply that $\vol(S^c \cap K_\delta^Q) \ge 7/16 - 1/16 = 6/16 > 1/4$. Since $d \le 1/2$ there exists $E \subset D, F \subset S^c \cap K_{\delta}^Q$ for which $\vol(E) = \vol(F) \ge d/2$. Let $E_0 \subset E$, $F_0 \subset F$ and $\Phi_1,\cdots,\Phi_m : E_0 \to K$ be the conull Borel sets and collection of Borel maps guaranteed by Lemma \ref{lem:deep-transport}. Because the maps satisfy \ref{item:core-equal-route},\ref{item:core-equal-steps} and \ref{item:deep-bounded-compression}, Lemma \ref{lem:first-exit} implies that
\[
\vol(\partial_0^K S) \ge  \frac{cd}{n(n+R/\delta)} \ge \frac{cdr}{n^2R} \implies d \le \frac{Cn^2R}{r}\vol(\partial_0^K S). 
\]
When $d = 0$ this upper bound on $d$ holds trivially. Plugging in the upper bound to the shell dichotomy \eqref{eq:shell-dichotomy} then gives 
\[
\vol(S) \le C \left[\frac{nR}{r}\vol(\partial_K^0 S) + \frac{n^2R}{r}\vol(\partial_K^0 S)\right] \implies \frac{\vol(\partial_0^K S)}{\vol(S)} \ge \frac{cr}{n^2R}.
\]
\end{proof}

\begin{proof}[Proof of Theorem~\ref{thm:main}]
For convenience we write $s = \vol(S)$ when it suits us. We take $H^*$ as defined by \eqref{eq:Hu} and \eqref{eq:u-star} and $\delta$ as in \eqref{eq:delta-core}.
Recall from Proposition \ref{prop:shell-reduction} the shell-dichotomy equation \eqref{eq:shell-dichotomy}
\[
\vol(S) \le \frac{CnR}{r}\vol(\partial_0^K S) + \vol(S \cap H^*).
\]
If $\vol(S \cap H^*) \le s/2$ then rearranging \eqref{eq:shell-dichotomy} yields
\[
\frac{\vol(\partial_0^K S)}{\vol(S)} \ge \frac{cr}{nR} \ge \frac{cr}{nR} \cdot \min \left\{1,\frac{\log(e/s)}{n}\right\}. 
\]
Thus we may assume that $\vol(S \cap H^*) > s/2$. If $\vol(S \cap H^*) \ge 1/96$ then $s > 1/96$ and by Proposition \ref{prop:linear-fallback} we have 
\[
\frac{\vol(\partial_0^K S)}{\vol(S)} \ge \frac{cr}{n^2R} \ge \frac{cr}{nR}\cdot \min \left\{1,\frac{\log(e/s)}{n}\right\}.
\]
Thus we may further assume that $\vol(S \cap H^*) < 1/96$. In this case we have by \eqref{eq:enough-complement} that $\vol(S^c \cap K_{\delta}^Q) \ge 7/16$. 
Therefore  
\[
\frac{\vol(S^c \cap K_{\delta}^Q)}{\vol(S \cap H^*)} \ge \frac{7}{16s} \ge \frac{1}{3s}.
\]
This means we can take conull Borel $E \subset S \cap H^*, F \subset S^c \cap K_{\delta}^Q$ and Borel maps $\Phi_1,\cdots,\Phi_M$ as guaranteed by Lemma \ref{lem:early-expansion} with $m = 1/(3s)$ and for some $\theta$ to be chosen later. Let $E_\theta = \{x \in E : \Phi_M(x) \in S\}$. Since $S$ is Borel and $\Phi_M$ is a Borel map the set $E_{\theta}$ is Borel. Therefore by \ref{item:core-unequal-expansion} of Lemma \ref{lem:early-expansion} we have 
\[
\frac{\vol(E)}{\vol(E_\theta)} = \frac{\vol(S \cap H^*)}{\vol(E_\theta)} \ge \frac{\vol(S)}{2\vol(E_\theta)} \ge \frac{\vol(\Phi_M(E_\theta))}{2\vol(E_\theta)} \ge \frac{1}{2}\left(\frac{1}{3s}\right)^{\theta} \implies \vol(E_\theta) \le 2(3s)^\theta \vol(E).
\]
In particular when $\theta \ge \log(4)/\log(1/(3s))$ it follows that $\vol(E_\theta) \le (1/2)\vol(E)$. Note also that $\log(4)/\log(1/(3s)) \le \log(4)/\log(16) = 1/2$. Since $\Phi_M(E_\theta^c) \subset S^c$ and satisfies \eqref{eq:early-bounded-compression} we may take $\theta = \log(4)/\log(1/(3s)) \le c/\log(1/s)$ and apply Lemma \ref{lem:first-exit} with $\Phi_1,\cdots,\Phi_M$ restricted to $E_\theta^c$ and conclude that 
\begin{align*}
\vol(\partial_0^K S) &\ge \frac{
  \vol(E_\theta^c)}{2M} \ge \frac{c}{M}\vol(S) \ge
   \min\left\{\frac{c}{n},\frac{cr}{n^2R\theta}\right\} \cdot \vol(S) \\
   &= \min\left\{\frac{c}{n},\frac{cr\log(1/s)}{n^2R}\right\}\cdot \vol(S) 
   \ge \frac{cr}{nR}\cdot \min \left\{1, \frac{\log(e/s)}{n}\right\} \cdot \vol(S),
\end{align*}
where the third inequality follows from \ref{item:core-unequal-steps}. In all cases we have shown that \eqref{eq:main-bound} holds and we conclude as desired.
\end{proof}
\begin{remark}[Continuous Compression]
Recall that the maps used in Lemmas \ref{lem:deep-transport} and \ref{lem:early-expansion} are such that the image of a point under two consecutive maps is the same except at possibly a single coordinate, and this coordinate is the same for all points. Using the fact that these maps $\Phi_{\ell,k}$ are injective and monotonicity of the Brenier map, one can show that whenever $x,y \in E$ are such that $\Phi_{\ell,k}(x)$ and $\Phi_{\ell,k}(y)$ agree on all but the $(k+1)$'st coordinate then 
$\sgn((\Phi_{\ell,k}(x) - \Phi_{\ell,k}(y))_{k+1}) = \sgn((\Phi_{\ell,k+1}(x) - \Phi_{\ell,k+1}(y))_{k+1})$. In particular the relative ordering of points along every coordinate fiber of $\Phi_{\ell,k}(E_0)$ is preserved by $\Phi_{\ell,k+1}$. In this way one can view the maps/coordinate moves as analogous to the  continuous compression operation (i.e. shaking). For context, compression is the operation whereby a set $S \subseteq K$ is transformed into another equal-volume subset of $K$ by partitioning $S$ along coordinate fibers and replacing the content of $S$ on a given fiber with an initial segment of said fiber having the same 1-dimensional Lebesgue measure. In view of discrete/combinatorial isoperimetric inequalities, the appearance of this idea as a proof method in our setting is not so surprising. (see \cite{Harper2004} for a survey). A compression normally replaces a set by a more ordered set of the same size
whose boundary is no larger, after which the ordered set can be analyzed.
In~\cite{Fernandez2026}, coordinate shaking is used in exactly this structural
way on the continuous cube.  Here the purpose is different: the shaking is used
dynamically to move labeled volume through the boundary and the resulting displacement yields the lower bound. 
\end{remark}
\section{Upper bounds}\label{sec:upper}

We begin this section by introducing all notation for the two constructions.  For an
unconditional convex body $Q\subset\R^n$, set
\[
 \sigma_i(Q)^2=\frac1{\vol(Q)}\int_Q x_i^2\,dx,\qquad
 \Lambda_Q=\operatorname{diag}\bigl(\sigma_1(Q)^{-1},\ldots,
                               \sigma_n(Q)^{-1}\bigr),
\]
and put $\widehat Q=\Lambda_QQ$.  Thus the uniform measure on $\widehat Q$ has
covariance $I$.  Write
\[
 q_Q=\|\one\|_{\widehat Q},\qquad
 \widehat a_Q=q_Q^{-1}\one.
\]
For any symmetric convex body $B$, the dual norm of a linear functional
$\varphi$ is
\[
 \|\varphi\|_{B^*}=\sup_{x\in B}|\varphi(x)|.
\]
Choose a norming functional $\widehat\ell_Q$ satisfying
\[
 \|\widehat\ell_Q\|_{\widehat Q^*}=1,\qquad
 \widehat\ell_Q(\widehat a_Q)=1,
\]
and define
\begin{equation}\label{eq:isotropic-diagonal-data}
 a_Q=\Lambda_Q^{-1}\widehat a_Q,\qquad
 \ell_Q=\widehat\ell_Q\circ \Lambda_Q,\qquad
 H_Q=\ker\ell_Q,\qquad F_Q=Q\cap H_Q.
\end{equation}
Such a norming functional
exists by the separating-hyperplane theorem
\cite[Theorem~1.1.5]{Zalinescu2002}. Notice that
$\|a_Q\|_Q=\|\ell_Q\|_{Q^*}=\ell_Q(a_Q)=1$.  

For $L,r>0$, define the diagonal $Q$-cylinder and diagonal $Q$-cone by
\begin{align}
 K_{L,r}^Q
 &= [-L,L]a_Q+rF_Q,\label{eq:diagonal-cylinder}\\
 K_{L,r}^{Q,\rm cone}
 &= \operatorname{conv}\left(\{0\},\,La_Q+rF_Q\right).
 \label{eq:diagonal-cone-intro}
\end{align}

The main tool used in verifying Proposition \ref{prop:capsule} is the following upper bound for the spacing between ordered entries of a typical point in $F_Q$.

\begin{lemma}\label{lem:spacing}
Let $n\ge 5$. Suppose $Q$ is an unconditional and isotropic convex body. If $V$ is uniformly distributed over
$F_Q$ and
$V_{(1)}\le\cdots\le V_{(n)}$ are its ordered coordinates, then there is a
deterministic index
\[
 j\in
 \left\{\left\lceil\frac n3\right\rceil,\ldots,
              \left\lfloor\frac{2n}3\right\rfloor\right\}
\]
such that
\begin{equation}\label{eq:spacing-bound}
 \mathbb E\bigl[V_{(j+1)}-V_{(j)}\bigr]\le\frac{C}{n}.
\end{equation}
For this index, one also has pointwise
\begin{equation}\label{eq:central-coordinate-bound}
 \max\{|V_{(j)}|,|V_{(j+1)}|\}\le C.
\end{equation}
\end{lemma}

\begin{proof}
Because $Q$ is unconditional isotropic it satisfies the following the containment inequality, which is due to Bobkov and Nazarov \cite[Propositions~2.4--2.5]{BobkovNazarov2003}

\begin{equation}\label{eq:uncond-sandwich}
\frac{1}{C}B_\infty^n \subset Q \subset CnB_1^n
\end{equation}
This immediately implies $(1/C) \le q_C \le C$. Indeed, $(1/C)B_\infty^n \subset Q$ implies that $\|(1/C)\ind\|_Q \le 1 = \|\hat{a}_Q\|_Q = \|q_Q^{-1}\ind\|_Q \implies q_Q \le C$ while $Q \subset CnB_1^n$ implies that $\|Cn(\ind/n)\|_Q \ge 1 = \|q_Q^{-1}\ind\|_Q \implies q_Q \ge (1/C)$.

Suppose $v \in F_Q$. Then by \eqref{eq:uncond-sandwich} $\|v\|_1 \le Cn\|v\|_Q \le Cn$. Consequently, the ordered entries of $v$ must satisfy

\begin{equation}\label{eq:coordinate-bounds}
-\frac{Cn}{j+1} \le v_{(j)} \le \frac{Cn}{n-j+1},
\end{equation}
otherwise the bound on $\|v\|_1$ is violated. In particular, since $V \in F_Q$ almost surely, this implies that
\[
C \ge \E[V_{\floor{2n/3}} - V_{\ceil{n/3}}] = \sum_{j = \ceil{n/3}}^{\floor{2n/3}-1} \E[V_{(j+1)} - V_{(j)}].
\]
In particular the average value of the summand is $C/n$ so there is some choice of $j$ for which the conclusion of Lemma \ref{lem:spacing} holds.
\end{proof}

We are now ready to prove Propositions~\ref{prop:capsule} and \ref{prop:cone}.

\begin{proof}[Proof of Proposition~\ref{prop:capsule}]
By scaling $Q$ by $\Lambda_Q$ it suffices to consider the case where $Q$ is isotropic. We suppress the subscripts on $q_Q$, $a_Q$,
and $\ell_Q$ here for convenience.  Choose the index $j$ from
Lemma~\ref{lem:spacing}. The cylinder has the parametrization 
\begin{equation}\label{eq:cylinder-parametrization}
 K_{L,r}^{Q}
 =\{tLa+rv: -1\le t\le 1,\ v\in F_Q\}.
\end{equation}
Let $J_{cyl}$ be the determinant Jacobian of the map $(t,v) \mapsto tLa + rv$. Note that the determinant is constant since the map is linear. Define the sign count $N(x)=\#\{i:x_i>0\}$ and the three classes

\begin{equation}\label{eq:capsule-cut}
 \begin{aligned}
 A =\{x \in K_{L,r}^Q : N(x) < n-j\},\qquad
 M =\{x \in K_{L,r}^Q : N(x) = n-j\},\qquad
 B =\{x \in K_{L,r}^Q : N(x) > n-j\}.
 \end{aligned}
\end{equation}

Clearly $A$ and $B$ are axis-disjoint and
\begin{equation}\label{eq:capsule-boundary-layer}
 \partial_0^{K_{L,r}^Q}A\subset M,\qquad
 \partial_0^{K_{L,r}^Q}B\subset M.
\end{equation}

For fixed $v$, with non-identical entries, the ordered entries of $vr+tLa$ are such that the $i$'th largest entry is 0 when $t = -rqv_{(i)}/L$. Consequently $vr+tLa$ is contained in $M$ exactly when $t \in [-rqv_{(j+1)}/L,-rqv_{(j)}/L)$. Therefore, since the set of $v$ with at least 2 identical entries has measure 0, by Fubini and Lemma \ref{lem:spacing} we have
\[
\vol(M) = \int_{F_Q} J_{cyl} \cdot \frac{rq}{L}(v_{(j+1)} - v_{(j)})~dx = J_{cyl} \cdot \vol_{n-1}(F_Q) \cdot \frac{rq}{L}\E[V_{(j+1)} - V_{(j)}] \le J_{cyl} \cdot \vol_{n-1}(F_Q) \cdot \frac{Crq}{nL}
\]
In addition by \eqref{eq:central-coordinate-bound} we know that $vr+tLa$ is contained in $A$ when $t \in [-1, -Cr/L]$ and is contained in $B$ when $t \in [Cr/L, 1]$. Therefore by Fubini we have 
\[
\min(\vol(A),\vol(B)) \ge J_{cyl}\cdot \vol_{n-1}(F_Q) \cdot (1-Cr/L)
\]
By taking $L \ge 2Cr$ we may assume that the denominator is at least $1/2$. 
Since $q$ is constant order the choice of $A,B,M$ imply that 
\[
\psi_0(K_{L,r}^Q) \le \frac{Cr}{nL}.
\]
It remains to verify $Q$-regularity. Containment in $(L+r)Q$ is clear since $\|v+ta\|_Q \le L+r$ for all $v \in rF_Q, t \in [-L,L]$. For showing containment of $(r/2)Q$ we do as follows: Let $P=I-a\ell$.
For any $z \in (r/2)Q$ we may write $z=\ell(z)a+Pz$. Now $|\ell(z)| \le \|\ell\|_{Q^*}\|z\|_Q \le r/2 \le L$ 
 and 
 $\|Pz\|_Q\le\|z\|_Q+|\ell(z)|\|a\|_Q
 \le2\|z\|_Q \le r$. In particular $Pz \in rF_Q$. Thus $(r/2)Q \subset K_{L,r}^Q$.
\end{proof}

\begin{proof}[Proof of Proposition~\ref{prop:cone}]
By scaling $Q$ by $\Lambda_Q$, it suffices to consider the
case where $Q$ is isotropic.  We suppress the subscripts on $q_Q$, $a_Q$,
and $\ell_Q$ for convenience. For convenience we write $K$ instead of $K_{L,r}^{Q,\rm cone}$
The cone has the parametrization
\begin{equation}\label{eq:cone-parametrization}
 K
 =\{tLa+rtv:0\le t\le1,\ v\in F_Q\}.
\end{equation}
Let $V$ be distributed uniformly on $F_Q$. For any non-negative integrable $f$ the average value of $f$ on the cone can be written as 
\begin{equation}\label{eq:cone-radial-disintegration}
\frac{1}{\vol(K)} \int_{K} f(x)~dx = n\int_0^1t^{n-1}
   \mathbb E\bigl[f(tLa+rtV)\bigr]\,dt 
   =n \E\left[ \int_0^1 t^{n-1}f(tLa + rtV)~dt \right]
\end{equation}

Define the sign count  $N_u(x)=\#\{i:x_i > uLa_i\}$ and the two classes
\[
 A_{u,j}=\{x\in K_{L,r}^{Q,\rm cone}:N_u(x) <  n-j\},\qquad M_{u,j}=\{x\in K_{L,r}^{Q,\rm cone}:N_u(x) = n-j\}
\]

By its construction $\partial^K_0 A_{u,j} \subseteq M_{u,j}$. We now determine, for suitable choices of $u$ and $j$, the range of $t$ for which $tLa + rtv$ is in $A_{u,j}$ and $M_{u,j}$. 

For convenience we write $\eps = qr/L$.
 Fix $v \in F_Q$. Then, if $1+\eps v_i > 0$, we have 
\begin{equation}\label{eq:threshold-t}
\frac{tL}{q} + rtv_i > \frac{Lu}{q} \iff t > \frac{Lu/q}{L/q + rv_i} = \frac{u}{1+\eps v_i}
\end{equation}

In particular $(tLa + rtv)_i > uLa_i$ exactly when $t$ is above the threshold in \eqref{eq:threshold-t}. Consequently we have
\begin{equation}\label{eq:threshold-set}
tLa + rtv \in 
\begin{cases}
  A_{u,j} & \text{ if } 0 \le t \le \frac{u}{1+ \eps v_{(j+1)}}\\
  M_{u,j} & \text{ if } \frac{u}{1+ \eps v_{(j+1)}} < t \le \frac{u}{1+ \eps v_{(j)}}\\
\end{cases}
\end{equation}
\\
Suppose now that $\ceil{n/3} \le j \le \floor{2n/3}$. To guarantee that $1+\eps v_{(j)} > 0$ we recall from the proof of Proposition~\ref{prop:capsule} that, since $Q$ is isotropic, it satisfies \eqref{eq:isotropic-deterministic-sandwich} and \eqref{eq:coordinate-bounds}. In particular $q \asymp 1$ and $|v_{(i)}| \le C$. Thus $1+\eps v_{(i)} > 0$ whenever $L/r$ is at least a sufficiently large absolute constant.
Moreover the lower bound on $L/r$ can be chosen so that $1+\eps v_{(j)} \in [1/2,3/2]$. Therefore when $u \le 1/2$ the thresholds for $t$ in \eqref{eq:threshold-set} are always at most 1. 

For the remainder of the proof we will assume $u \le 1/2$.

We now bound $\vol(M_{u,j})/\vol(A_{u,j})$ and $\vol(A_{u,j})/\vol(K_{L,r}^{Q,cone})$. Taking $f$ to be the appropriate indicator functions in 
\eqref{eq:cone-radial-disintegration} and using \eqref{eq:threshold-set} we have 
\begin{equation}\label{eq:cone-ratio}
\frac{\vol(M_{u,j})}{\vol(A_{u,j})} = 
\frac{
\E\left[\left(\frac{u}{1+\eps v_{(j)}}\right)^n\right] - \E\left[\left(\frac{u}{1+\eps v_{(j+1)}}\right)^n\right]
}{\E\left[\left(\frac{u}{1+\eps v_{(j+1)}}\right)^n\right]} = \frac{\E[(1+\eps v_{(j)})^{-n}]}{\E[(1+\eps v_{(j+1)})^{-n}]} - 1.
\end{equation}
Similarly, one has $\vol(A_{u,j})/\vol(K_{L,r}^{Q,cone}) = u^n\E[(1+\eps v_{(j+1)})^{-n}] \in [(2u/3)^n,(2u)^n]$. Thus whenever $s^{1/n}$ is at most a sufficiently small constant and $\ceil{n/3} \le j \le \floor{2n/3}$ there is a choice of $u \le 1/2$ for which $\vol(A_{u,j})/\vol(K_{L,r}^{Q,cone}) = s$.

Returning to \eqref{eq:cone-ratio} we note that 

\[
\left(\frac{(1 + Cr/L)}{(1-Cr/L)}\right)^n \ge \frac{\E[(1+\eps v_{(\ceil{n/3})})^{-n}]}{\E[(1+\eps v_{(\floor{2n/3})})^{-n}]} \ge \prod_{j = \ceil{n/3}}^{\floor{2n/3}-1}\frac{\E[(1+\eps v_{(j)})^{-n}]}{\E[(1+\eps v_{(j+1)})^{-n}]}
\]
Thus there is some $\ceil{n/3} \le j \le \floor{2n/3}$ for which $\frac{\E[(1+\eps v_{(j)})^{-n}]}{\E[(1+\eps v_{(j+1)})^{-n}]}$ is at most $\left(\frac{(1+Cr/L)}{(1-Cr/L)}\right)^7$. By taking $L/r$ to be at least a sufficiently large absolute constant this quantity can be further bounded by $1 + Cr/L$. In particular, for this choice of $j$, \eqref{eq:cone-ratio} yields $\frac{\vol(M_{u,j})}{\vol(A_{u,j})} \le (Cr/L)$. 

We are left with showing that $K_{L,r}^{Q,cone}$ satisfies the necessary regularity conditions.
$K_{L,r}^{Q,\rm cone}\subset(L+r)Q$ follows directly from
\eqref{eq:cone-parametrization}. To show inclusion of a ball we do as follows:  If $z\in(r/8)Q$, write
\[
 \frac L2a+z=tLa+Pz,\qquad
 t=\frac12+\frac{\ell(z)}L,
\]
where $P=I-a\ell$.  Since $L\ge r$,
$3/8\le t\le5/8$, and
$\|Pz\|_Q\le\|z\|_Q+|\ell(z)|\|a\|_Q\le2\|z\|_Q$ gives
\[
 \left\|\frac{Pz}{rt}\right\|_Q\le\frac23.
\]
As $Pz\in H$, the parametrization shows that
$La/2+(r/8)Q\subset K_{L,r}^{Q,\rm cone}$.  
\end{proof}

\section{Closing Remarks}\label{sec:remarks}
To conclude we would like to draw the reader's attention to two follow-up questions to the work presented in this paper (specifically that of the lower bound).
\begin{enumerate}
\item \label{item:question-1} \emph{An Optimal $\ell_0$-isoperimetric inequality}: Under the assumptions of Theorem \ref{thm:main} it is natural to ask if \eqref{eq:set-main} can be further improved. We believe that the worst case should be achieved by the convex body in Theorem \ref{thm:upper}, in which case \eqref{eq:set-main} is off by a factor of $n$. The author spent quite a bit of effort trying to improve \eqref{eq:set-main} by improving the Brenier transport analysis but was unsuccessful. In particular we tried to show that one could take a coarser coordinate-staircase for the Brenier transport. Although we showed that we could do so while maintaining a similar bounded-compression property (in some average case sense) we could not argue that the coarser transport wouldn't force most staircases to partially leave the body, invalidating the analysis. Although our use of a global analysis gives an improvement over those provided in \cite{LaddhaVempala2023,NarayananRajaramanSrivastava2025}, each of which involves some sort of local analysis, we believe a further local analysis to the transport approach would be needed to see an improved $\ell_0$ lower bound.
\item \label{item:question-2} \emph{Further improving $\ell_0$ isoperimetry for unconditional convex bodies}: As was mentioned in \ref{item:question-1} we believe that the lower bound in Theorem \ref{thm:main} can be improved by an additional factor of $n$ and this would match the bound appearing in Theorem \ref{thm:upper}. However, note that Theorem \ref{thm:upper} requires that the ratio between outer and inner regularity be at least a sufficiently large constant. This might seem like a proof artifact and one might even speculate that Theorem \ref{thm:upper} could hold with $R = r$. However, perhaps surprisingly, that is not possible. Indeed, in \cite{Fernandez2026} it was shown that the $\ell_0$ isoperimetric coefficient of an axis-aligned cube is $\Theta(n^{-1/2})$ and that the boundary ratio for any subset is at least of order $n^{-1/2}$. Furthermore the lower bound on the isoperimetric profile of a binomial random variable shown in \cite{Fernandez2026} actually implies that the boundary ratio for subsets of the cube go to $\infty$ as the subset's volume goes to 0. It is an interesting question to ask what happens to the worst case behavior of $\ell_0$ boundary for $Q$-regular convex bodies as the inner and outer regularity converge to one another. In \cite{Fernandez2026} it was speculated that the symmetrization technique of Harper \cite{Harper1999} for Hamming cubes could be modified in a suitable way so as to be applicable to an isotropic unconditional convex body. If doable it seems likely that the resulting analysis would give a boundary ratio exceeding what would be implied by the best possible improvement to Theorem \ref{thm:main}.
\end{enumerate}   

\section{Appendix}
\subsection{Proof of Lemma \ref{lem:shell-path}}\label{sec:shell-path proof}
\begin{proof}
Let
\[
 L=\left\lceil\frac{120R}{r}\right\rceil,
 \qquad \delta=\frac1L,
 \qquad t_\ell=\ell\delta,\qquad 0\le\ell\le L.
\]
Let $P_k$ denote projection onto the first $k$ coordinates. Let $T$ denote the radial map from \eqref{eq:radial-map}. For convenience we suppress the subscript $K$ in $\rho_K$. For
$0\le\ell<L$, $0\le k\le n$, and $x\in E$ define
\begin{equation}\label{eq:diagonal-hybrid}
 \Psi_{\ell,k}(x)
 =(1-t_\ell)x
   +t_\ell T(x)
   +\delta\left(1-\frac{1}{\rho(x)}\right)P_kx.
\end{equation}
We now check that $\Phi_{\ell,k}$ satisfies the necessary properties from \ref{lem:shell-path}.

\begin{enumerate}[label=(\roman*)]
  \item From \eqref{eq:diagonal-hybrid} it's immediate that $\Psi_{0,0} = \Id$ and $\Psi_{L-1,n} = T$. In addition $\Psi_{
    \ell,k}$ and $\Psi_{\ell,k+1}$ can only defer on the $(k+1)$'st coordinate. This proves \ref{item:shell-route}. 
  \item $M \le Ln \le \ceil{\frac{120R}{r}}n \le \frac{CnR}{r}$
  \item Fix $\ell,k$ and, for a given $x \in E$, let $y = \Psi_{\ell,0}(x)$.  By \eqref{eq:diagonal-hybrid} and homogeneity of $\rho$, one has 
  \begin{align*}
  \rho(y) &= (1-t_\ell)\rho(x) + \left(2 - \frac{1}{\rho(x)}\right)\rho(x)t_\ell \\
  &= \rho(x) + t_\ell\rho(x)-t_\ell \\
  &= (1+t_\ell)(\rho(x)-1)+1 \\
  \implies \rho(x)-1 &= \frac{\rho(y)-1}{1+t_\ell}.
  \end{align*}
  Therefore 
  \[
  \Psi_{\ell,k}(x) = y + \delta\left(1 - \frac{1}{\rho(x)}\right)P_kx = y + \delta\left(\frac{\rho(x)-1}{\rho(y)}\right)P_ky = y + \frac{\delta}{1+t_\ell}\left(1-\frac{1}{\rho(y)}\right)P_ky. 
  \]
  Now suppose $x' \in E$ is distinct from $x$ and $y' = \Psi_{\ell,0}(x')$. Since $x,x' \in E$ they satisfy $\rho(x),\rho(x'),\rho(y),\rho(y') \in [1/2,1]$. Therefore 
  \begin{align*}
    \|\left(1 - \frac{1}{\rho(y)}\right)P_ky - \left(1 - \frac{1}{\rho(y')}\right)P_ky'\|_Q 
    &\le \left|\left(1 - \frac{1}{\rho(y)}\right)\right|\|P_k(y-y')\|_Q + \left|\frac{1}{\rho(y)} - \frac{1}{\rho(y')}\right|\|P_ky'\|_Q \\
    &\le \|y-y'\|_Q + |\rho(y)-\rho(y')| \frac{\|P_ky'\|_Q}{\rho(y)\rho(y')} \\
    &\le \frac{10R}{r}\|y-y'\|_Q,
  \end{align*}
  where on the last line we used \eqref{eq:minkowski-Lipschitz} and \eqref{eq:diam-bound}.
  Therefore 
  \[
  \|\Psi_{\ell,k}(x) - \Psi_{\ell,k}(x')\|_Q \ge \|y-y'\|_Q - \frac{\delta}{1 + t_\ell}\frac{10R}{r}\|y-y'\|_Q \ge \frac{11}{12}\|y-y'\|_Q > 0,
  \]
  where the second to last inequality follows from the definition of $\delta$ and the last inequality follows from $\Psi_{\ell,0}$ being an injection. Thus $\Psi_{\ell,k}$ is indeed an injection on $E$. To see that it's Lipschitz, note that the argument for injectivity also implies that 
\[
\|\Psi_{\ell,k}(x) - \Psi_{\ell,k}(x')\|_Q \le \frac{13}{12}\|y - y'\|_Q = \frac{13}{12}\|\Psi_{\ell,0}(x) - \Psi_{\ell,0}(x')\|_Q.
\] Since the identity map and radial-map are Lipschitz maps and $\Psi_{\ell,0}$ is a convex combination of the two we conclude that $\Psi_{\ell,0}$ is Lipschitz. It immediately follows that $\Psi_{\ell,k}$ is Lipschitz.
  \item Recall from the proof of \ref{lem:shell-doubling} that we computed
  \[
  D \left(\left(2 - \frac{1}{\rho(x)}\right)x\right) = \left(2 - \frac{1}{\rho(x)}\right) \cdot I + \frac{1}{\rho(x)^2}x(\nabla \rho(x))^\top. 
  \]
  Similarly, one can compute 
  \[
  D \left(\left(1 - \frac{1}{\rho(x)}\right)P_kx\right) = \left(1 - \frac{1}{\rho(x)} \right) \cdot P_k + \frac{1}{\rho(x)^2}(P_k x)(\nabla \rho(x))^\top 
  \]
  Therefore we have 
  \begin{align*}
  D\Psi_{\ell,k}(x) &= D\left((1-t_\ell)x\right)
   +D\left(t_\ell T(x)\right)
   +D\left(\delta\left(1-\frac{1}{\rho(x)}\right)P_kx\right) \\
   &=\underbrace{\left(1+ t_\ell- \frac{t_\ell}{\rho(x)}\right)\cdot I + \left(\delta - \frac{\delta}{\rho(x)}\right)P_k}_{=:A} + \underbrace{\frac{1}{\rho(x)^2}(t_\ell x + \delta P_kx)}_{=: u}\underbrace{(\nabla \rho(x))^\top}_{=: v^\top}.
  \end{align*}
  To use the determinant formula for a rank-one perturbation we compute $\det(A)$ and $1 + v^\top A^{-1}u$. Since $x \in E$ we have 
  \begin{align*}
  \det(A) 
  &= \left(1 + t - \frac{t}{\rho}\right)^{n-k}\left(1 + t + \delta - \frac{t + \delta}{\rho}\right)^{k} \\
  &= \left(1 + t\left(1 - \frac{1}{\rho}\right)\right)^{n-k}\left(1 + (t+\delta)\left(1 - \frac{1}{\rho}\right)\right)^{k} \\
  &\ge \left(1 - \left(\frac{1/(8n)}{1-1/(8n)}\right)\right)^{n-k}\left(1 - \left(\frac{(1+\delta)/(8n)}{1-1/(8n)}\right)\right)^{k} \\
  &\ge 1 - \frac{(1+\delta)/8}{1-1/8} \ge 5/7
\end{align*}
where for the last inequality we used the fact that $\delta \leq 1$.
For the next step it will be convenient to define the following values 
\[
a := 1 + t_\ell(1 - 1/\rho(x)),\qquad b = \delta(1 - 1/\rho(x)), \qquad c = t_{\ell}/\rho(x)^2,\qquad d = \delta/\rho(x)^2.
\] 
We can write 
\[
A = aI + bP_k, \qquad A^{-1} = \underbrace{a^{-1}I}_{=: B_1} + \underbrace{((a+b)^{-1} - a^{-1})P_k}_{=: B_2}, \qquad u = cx + dP_kx.
\]
Then 
\begin{equation}\label{eq:shell-det-technical}
  \begin{split}
v^\top A^{-1} u 
&= (\nabla \rho(x))^\top(B_1+B_2)(cx + dP_kx) \\
&= a^{-1}\ip{cx}{\nabla \rho(x)} + \ip{A^{-1}(dP_k x)}{\nabla \rho(x)} + \ip{B_2 cx}{\nabla\rho(x)}
  \end{split}
\end{equation}
We now bound the terms of \eqref{eq:shell-det-technical} separately. The first term satisfies 
\[
a^{-1}\ip{cx}{\nabla \rho(x)} = \frac{a^{-1}t_\ell}{\rho(x)} \ge 0.
\]
The second term satisfies 
\[
|\ip{A^{-1}(d P_k x)}{\nabla \rho(x)}| \le \|A^{-1}\|_{op} \cdot d \cdot \|P_k x\|_Q \cdot \|\nabla \rho(x)\|_{Q^*} \le \frac{2R (a+b)^{-1}\delta}{r\rho(x)^2} \le \frac{4R\delta}{r} \le \frac{1}{30}.
\]
The third term satisfies 
\[
|\ip{B_2 cx}{\nabla\rho(x)}| \le \|B_2\|_{op} \cdot \|cx\|_Q \cdot \|\nabla \rho(x)\|_{Q^*} \le |(a+b)^{-1}-a^{-1}| \cdot \frac{2R}{\rho(x)^2} \cdot \frac{1}{r} \le
\frac{|b|}{|a(a+b)|} \cdot \frac{2R}{r\rho(x)^2} \le \frac{4R\delta }{r} \le \frac{1}{30}
\]
We thus deduce from the determinant formula that 
\[
\det(D\Psi_{\ell,k}(x)) = \det(A)(1 + v^\top A^{-1} u) \ge \frac{5}{7}\left(1 - \frac{1}{15}\right) > \frac{1}{2}.
\]
Since $\Psi_{\ell,k}$ is Lipschitz and injective we may apply Lemma \ref{lem:lipschitz-area-formula} with the lower bound of $\det D\Psi_{\ell,k}$ on $E$ to conclude \eqref{eq:shell-bounded-compression}.
\item Fix $\ell,k$, let $x \in E$ and let $y = \Psi_{\ell,0}(x)$. In the proof of \ref{item:shell-injective} we argued that 
\[
\eta(y) = 1-\rho(y) = (1+t_\ell)(1-\rho(x)) = (1+t)\eta(x) \le 2\eta(x)
\]
Therefore from \eqref{eq:shell} we know that $u \le \eta(y) < 4u$. Furthermore
\[
|\rho(\Psi_{\ell,k}(x)) - \rho(y)| \le \frac{1}{r}\|\Psi_{\ell,k}(x) - y\|_Q \le \frac{2\delta\eta(x)}{r\rho(x)}\|P_k x\|_Q \le\frac{4\delta R u}{r(1 - 1/(16n))} \le \frac{u}{2}.
\]
Therefore $1/2 \le \rho(\Psi_{\ell,k}(x)) \le 5/2$ and $\Psi_{\ell,k}(x)$ lies in $U_u$.
\end{enumerate}
\end{proof}

\subsection{Proof of Lemma \ref{lem:deep-transport}}\label{sec:deep-transport proof}
\begin{proof}
Take $\mu$ and $\nu$ to be the probability measures induced by the normalized Lebesgue measures on $E$ and $F$ respectively. By
Theorem~\ref{thm:brenier}, there are conull Borel sets $E_0\subset E$,
$F_0\subset F$, a proper convex function $\varphi$ and a Borel bijection
$\bren:E_0\to F_0$ whose inverse is
$\bren^*$. Since $\varphi$ is convex, Aleksandrov's theorem \cite{Aleksandrov1939} gives the existence of $\nabla^2 \varphi$ $\mu$.a.e. in int(dom $\varphi$). In addition $\det \nabla^2 \varphi$ is locally integrable (see the proof of Lemma \ref{lem:monotone-image-volume}) which means $\mu$.a.e point of $E$ is a Lebesgue point of $\det \nabla^2\varphi$. Lastly since $\nabla^2\varphi$ satisfies \eqref{eq:det-one}, it is invertible $\mu$.a.e. Thus we may further restrict $E_0,F_0$ to co-null Borel subsets on which \eqref{eq:monotonicity}, \eqref{eq:det-one} hold and can take $M = E_0$ for $\varphi$ in Lemma \ref{lem:monotone-image-volume}.
Next, consider the interpolation map $X_t : E_0 \to K^Q_{u_*r}$ defined by $X_t = (1-t)I + t\nabla \varphi$. The image of the map is indeed in $K_{u_*r}^Q$ by convexity of $K_{u_*r}^Q$. Next, 
let
 \[
 L = \ceil{64nR/r},\qquad \delta = \frac{1}{L}, \qquad t_\ell = \ell \delta,\qquad 0 \le \ell \le L.
 \] 

 Let $P_k$ denote projection onto the first $k$ coordinates. For
$0\le\ell<L$, $0\le k\le n$, and $x\in E_0$ define
\begin{equation}\label{eq:brenier-hybrid}
 \Phi_{\ell,k}(x)
 =(1-t_\ell)x
   +t_\ell \bren(x)
   +\delta P_k(\bren(x)-x).
\end{equation}
We now check that $\Phi_{\ell,k}$ satisfy the necessary properties from \ref{lem:deep-transport}.
\begin{enumerate}[label=(\roman*)]
  \item By construction $\Phi_{(0,0)}(x) = x$, $\Phi_{L-1,n}(x) = \bren(x)$, and, for every $x \in E_0$, $\Phi_{\ell,k}(x)$ and $\Phi_{\ell,k+1}(x)$ can only differ on the ($k+1$)'st coordinate (treating $k+1$ as $1$ when $k = n$). In addition $\delta P_k(\bren(x)-x) \in u_*rQ$ and $X_t(x) \in K_{u_*r}^Q$ so $\Phi_{\ell,k}(x) \in K$. Therefore $\Phi_{\ell,k}$ is indeed a map from $E_0$ to $K$. Altogether this proves \ref{item:core-equal-route}.
  \item $M \le Ln \le \ceil{64nR/r}\cdot n \le Cn^2R/r$.
  \item Since $\Id$ and $\bren$ are injections on $E_0$ it suffices to consider the case where $(\ell,k)$ is neither $(0,0)$ nor $(L-1,n)$ Fix such a choice of $\ell$ and $k$. For convenience we write $\Phi$ for $\Phi_{\ell,k}$. Let $x,y \in E_0$ and suppose $\Phi(x) = \Phi(y)$. Assume for the sake of contradiction that $x \neq y$. From \eqref{eq:brenier-hybrid} we have 
  \begin{align*}
  0 =\Phi(x)-\Phi(y) = (1-t_\ell)(x-y) + t_\ell(\bren(x)-\bren(y)) + \delta P_k(\bren(x)-\bren(y)-(x-y)) = 0 \\
  \implies (t_\ell I + \delta P_k)(\bren(x)-\bren(y)) = - ((1-t_\ell)I - \delta P_k)(x-y).
  \end{align*}
For standard basis vector $e_i$ we have 
\begin{equation}\label{eq:brenier-inject-cond}
(t_\ell I + \delta P_k)e_i = 0 \iff t_\ell = 0,i > k \qquad \text{and} \qquad ((1-t_\ell)I - \delta P_k)e_i = 0 \iff t_\ell = 1-\delta,i \le k
\end{equation}
Thus there are 3 cases:
\begin{enumerate}[label=(\roman*)]
  \item \label{item:case-1}The condition $(t_\ell I + \delta P_k)e_i = 0$ forces  
  \[ (x-y)_i = 0 \implies (x-y)_i \cdot (\bren(x)-\bren(y))_i = 0
  \]
  \item \label{item:case-2}The condition $((1-t_\ell)I - \delta P_k)e_i = 0$ forces  
  \[(\bren(x)-\bren(y))_i = 0 \implies (x-y)_i \cdot (\bren(x)-\bren(y))_i = 0 
  \]
  \item \label{item:case-3} If either the condition for \ref{item:case-1} is not satisfied and $(x-y)_i \neq 0$ or the condition for \ref{item:case-2} is not satisfied and $(\bren(x)-\bren(y))_i \neq 0$ then 
  \[0 \neq \sgn((x-y)_i) = -\sgn((\bren(x)-\bren(y))_i) \implies (x-y)_i \cdot (\bren(x)-\bren(y))_i < 0
  \]
\end{enumerate}
Next, since $\Id$ and $\bren$ are injections on $E_0$ there exists indices $j_1,j_2$ such that $(x-y)_{j_1} \neq 0$ and $(\bren(x)-\bren(y))_{j_2} \neq 0$.
Now because of the constraint on $t_\ell$ in \eqref{eq:brenier-inject-cond} 
the conditions for 
\ref{item:case-1} and
\ref{item:case-2}
can't both be satisfied.
If \ref{item:case-1} can't be satisfied then, since $(x-y)_{j_1} \neq 0$, \ref{item:case-3} occurs and if \ref{item:case-2} can't be satisfied then, since $(\bren(x)-\bren(y))_{j_2} \neq 0$, \ref{item:case-3} also occurs. In either case we deduce that $\ip{\bren(x)-\bren(y)}{x-y} < 0$. But $\bren$ is monotone (see \eqref{eq:monotonicity}) meaning $\ip{\bren(x)-\bren(y)}{x-y} \ge 0$. As this is a contradiction we conclude that $x = y$. Therefore $\Phi$ is an injection. 
  \item 
  The proof of \ref{item:deep-bounded-compression}
has two parts. The first part is to argue that the determinant of the Jacobian of $\Phi_{\ell,k}$ is at least $1/2$ for all coordinate hybrids $\Phi_{\ell,k}$ and at all $x \in E_0$.
Write
\begin{align*}\label{eq:hybrid-Jacobian}
  D\Phi_{\ell,k}(x) &= D((1-t_\ell)x) + D(t_\ell(\bren(x))) + D(\delta P_k(\bren(x)-x)) \\
  &= (1-t_\ell)I + t_{\ell}\hess(x) + \delta P_k(\hess(x) - I) \\
  &= I + (t_\ell I + \delta P_k)(\hess(x)-I)
\end{align*}
where $\hess(x)$ denotes the Hessian of $\varphi$ at $x$. Because $\hess(x)$ is symmetric and positive semi-definite it admits an eigendecomposition
\begin{equation}\label{eq:ev-decomp-H}
\hess(x) = \sum_{i = 1}^n \lambda_i v_iv_i^\top, 
\end{equation}
where each $\lambda_i \ge 0$ is a non-negative eigenvalue of $\hess(x)$ and each $v_i$ are orthonormal eigenvectors of $\hess(x)$.

When $k = 0$ the map $P_k$ is identically 0 so, in view of \eqref{eq:ev-decomp-H} we may write
\begin{equation}\label{eq:ev-decomp-common}
D\Phi_{\ell,0}(x) = I + t_\ell(\hess(x)-I) = \sum_{i = 1}^n (1-t_{\ell} + \lambda_i t)v_iv_i^\top
\end{equation}

Therefore 
\[
\det D\Phi_{\ell,0}(x) = \prod_{i = 1}^n (1-t_{\ell} + \lambda_i t_{\ell}) \ge \prod_{i = 1}^n \lambda_i^{t_{\ell}} = (\det \hess(x))^{t_{\ell}} = 1,
\]
where the middle inequality is AM-GM and the last equality follows from \eqref{eq:det-one}.

When $k \neq 0$ more effort is required. 
We write $[k] = \{1,\cdots,k\}, [k]^c = \{1,\cdots,n\}\setminus [k]$ and write $M_{J_1,J_2}$ to denote the submatrix of $M$ whose rows and columns are indexed by $J_1$ and $J_2$ respectively.
Suppose $\ell$ satisfies $t_{\ell} \le 1/2$. For convenience we write $U_1 := D\Phi_{\ell,0}(x),U_2 := D\Phi_{\ell,k}(x)$ and let $J = \{1,\cdots,k\}$. We first bound $\det(U_2)/\det(U_1)$ from below.  $U_1$ is invertible since $\det(U_1) \ge 1$. Therefore we may write
\begin{equation}\label{eq:prod-ident}
U_2U_1^{-1} = (U_1 + (U_2-U_1))U_1^{-1} = I + \delta P_k(\hess(x)-I)U_1^{-1},
\end{equation}
Up to relabeling of coordinates we may write $U_1^{-1}$ as the block matrix

\[
U_1^{-1} =
\begin{pmatrix}
(U_1^{-1})_{[k],[k]} & (U_1^{-1})_{[k],[k]^c} \\
(U_1^{-1})_{[k]^c,[k]} & (U_1^{-1})_{[k]^c,[k]^c}
\end{pmatrix}
\]
Consequently 
\[
I + \delta P_{k}(\hess(x)-I)U_1^{-1} = 
\begin{pmatrix}
I_{[k],[k]} + \delta((\hess(x)-I)U_1^{-1})_{[k][k]} & \delta P_k(\hess(x))(U_1^{-1})_{[k],[k]^c} \\
0 & I_{[k]^c,[k]^c}
\end{pmatrix}
\]
In particular the matrix is upper-block triangular so its determinant is the product of the determinants of the diagonal blocks. 
The bottom right block is identity and has determinant 1. For the top left block we may use \eqref{eq:ev-decomp-H} and \eqref{eq:ev-decomp-common} to write 

\begin{equation}\label{eq:ev-decomp-1}
I + \delta((\hess(x)-I)U_1^{-1}) = \sum_{i = 1}^n \left(1 + \frac{\delta(\lambda_i - 1)}{(1 - t_{\ell} + \lambda_i t_{\ell})}\right)v_iv_i^\top.
\end{equation}

A routine computation shows that the map $\lambda \mapsto (\lambda - 1)/(1 - t_{\ell} + \lambda t_{\ell})$ has non-negative derivative whenever $1 + t_\ell(\lambda - 1)$ is non-zero, and this holds whenever $\lambda \ge 0$ and $t_{\ell} \in [0,1/2]$. In particular the scalars appearing in the summand in \eqref{eq:ev-decomp-1} are between $1 - \delta/(1-t_{\ell})$ and $1+\delta/t_{\ell}$ (which is $+\infty$ when $t_\ell = 0$) and so the eigenvalues of $I + \delta (\hess(x)-I)U_1^{-1}$ fall in this range. By the Cauchy-Interlacing theorem the eigenvalues of $I_{[k],[k]} + \delta((\hess(x)-I)U_1^{-1})_{[k][k]}$ also fall in this range. In particular its determinant is at least $(1-2\delta)^n > (1-1/(2n))^n \ge 1/2$.
Combining \eqref{eq:prod-ident} and our block determinant estimates we conclude that 
\[
\det(U_2) = \det(U_1) \cdot \det(U_2U_1^{-1}) \ge 1 \cdot \frac{1}{2} = \frac{1}{2}
\]
Suppose now $\ell$ is such that $t_\ell \ge 1/2$. The analysis is instead repeated with $U_2 = D\Phi_{\ell,k},U_1 = D\Phi_{\ell+1,0}$. It is nearly identical to the previous case except for some minor differences. In particular the rightmost expression in \eqref{eq:prod-ident} becomes $I - \delta (I-P_k)(\hess(x)-I)U_1^{-1}$ and the corresponding version of \eqref{eq:ev-decomp-1} implies that the eigenvalues of $I - \delta (I-P_k)(\hess(x)-I)U_1^{-1}$ are between $1 - \delta/t_{\ell + 1}$ and $1 + \delta/(1-t_{\ell+1})$ (which is $+\infty$ when $t_{\ell+1} = 1$). The exact same determinant argument as in the previous case then implies that $\det(U_2) \ge 1/2$. This completes the first part.
  
We now turn our attention to the second part (i.e. converting the determinant lower bound to bounded compression).
For convenience we write $\Phi$ for $\Phi_{\ell,k}$. 
Note that
\[
\Phi_{\ell,k}(x) = (1-t_\ell)x + t_\ell \bren(x)+\delta P_k(\bren(x)-x) = \underbrace{((1-t_\ell)I - \delta P_k)}_Ax + \underbrace{(t_\ell I + \delta P_k)}_B(\bren(x)).
\]
We case on $\ell$.
\begin{itemize}
  \item $\ell > 0$: In this case $B$ is invertible. Therefore $B^{-1}\Phi(x) = B^{-1}Ax + \bren(x)$.
  Define 
  \begin{equation}\label{eq:brenier-hybrid-map-1}
  \psi(x) = \frac{1}{2}x^\top B^{-1}Ax + \varphi(x).
  \end{equation}
  Since $B^{-1}A$ is a non-negative diagonal matrix the map $x \mapsto \frac{1}{2}x^\top B^{-1}Ax$ is convex. Therefore, since $\varphi$ is also convex, so is $\psi$. From \eqref{eq:brenier-hybrid-map-1} it follows that
  \begin{equation}\label{eq:brenier-hybrid-map-2}
    \begin{split}
  \nabla \psi(x) &= B^{-1}Ax + \nabla \varphi(x) = B^{-1}\Phi(x)\\
  \nabla^2\psi(x) &= B^{-1}A + \hess(x) = DB^{-1}\Phi(x) \\
    \end{split}
  \end{equation}
  This means $B^{-1}\Phi$ is the gradient of a convex function.
  After further restricting to conull Borel $E'_0 \subset E_0$, on which $\nabla^2 \psi$ exists and $\det \nabla^2 \psi$ has all Lebesgue points, we have for every Borel $G \subset E_0$
  
  \begin{align*}
  \vol(\Phi(G)) &= \vol(BB^{-1}\Phi(G)) \\
  &= \det(B)\vol(B^{-1}\Phi(G)) \\
  &\ge \det(B)\vol(\nabla \psi(G \cap E'_0)) \\
  &\ge \det(B)\int_{G \cap E'_0} \det \nabla^2 \psi(x) dx \\
  &= \det(B) \int_{G \cap E'_0} \det DB^{-1}\Phi(x)~dx \\
  &= \det(B) \int_{G \cap E'_0} \det B^{-1}D\Phi(x)~dx \\
  &= \int_{G \cap E'_0} \det D\Phi(x)~dx \\
  &= \int_{G} \det D\Phi(x)~dx
  \ge \frac{1}{2}\vol(G).
  \end{align*}
  On the second and last and second to last equalities we used the fact that $B$ is a linear map. For the third and fourth equalities we used \eqref{eq:brenier-hybrid-map-2} and for the first inequality we used Lemma \ref{lem:monotone-image-volume}.
  \item $\ell = 0$: In this case $A = I-\delta P_k$ and $B = \delta P_k$. If $k = 0$ then $\Phi = \Id$ and there's nothing to do so we assume $k > 0$. For $x \in \R^n$ we write $x = (w,z) \in \R^k \times \R^{n-k}$. For $z \in \R^{n-k}$ define $E_{0,z} = \{w \in \R^k : (w,z) \in E_0\}$. Since $E_0$ is Borel so is $E_{0,z}$ for every $z$. Define the map
  \[
  \Phi_z(w) = (1-\delta)w + \delta \nabla_w \varphi(w,z).
  \]
  Its Jacobian is 
  \[
  D_w \Phi_z(w) = (1-\delta)I_k + \delta(\hess(w,z))_{[k],[k]}, 
  \]
  Which is exactly the top $k \times k$ left block of $D\Phi_{0,k}(w,z)$. Therefore its determinant is at least 1/2.
  Note that $\Phi_z(w)$ is the gradient of 
  \[
  \psi_z(w) = \frac{1-\delta}{2}\|w\|_2^2 + \delta \varphi(w,z).
  \]
  which itself a convex function. Therefore for each $E_{0,z}$ we may restrict to a conull Borel subset $E'_{0,z}$ on which $\nabla_w^2 \psi_z$ exists and $\det \nabla_w^2 \psi_z$ has all Lebesgue points. Therefore by Lemma \ref{lem:monotone-image-volume} all Borel $G_z \subset E_{0,z}$ satisfy 
  \[
  \vol(\Phi_z(G_z)) \ge \vol(\Phi_z(G_z \cap E'_{0,z})) \ge \frac{1}{2}\vol(G_z \cap E'_{0,z}) = \frac{1}{2}\vol(G_z \cap E_{0,z}) = \frac{1}{2}\vol(G_z).
  \] 
  for every $z$.
  Lastly for every Borel $G \subset E_0$ let $G_z := \{w \in \R^k : (w,z) \in G\}$. By Fubini we have 
  \[
  \vol(\Phi(G)) = \int_{\R^{n-k}}\vol(\Phi_z(G_z))~dz \ge \frac{1}{2}\int_{\R^{n-k}}\vol(G_z)~dz = \frac{1}{2}\vol(G),
  \]
  
\end{itemize}
\end{enumerate}
\end{proof}

\bigskip
\noindent
Department of Mathematics, University of Southern California,\newline
3620 S. Vermont Ave., Kaprielian Hall (KAP 104),
Los Angeles, CA 90089-2532, USA

\smallskip
\noindent
Email address: \href{mailto:manuelf7@usc.edu}{\texttt{manuelf7@usc.edu}}


\clearpage
\begin{thebibliography}{99}

\bibitem{AscolaniLavenantZanella2024}
F.~Ascolani, H.~Lavenant, and G.~Zanella,
\newblock Entropy contraction of the Gibbs sampler under log-concavity,
\newblock arXiv:2410.00858, revised 2026.

\bibitem{BobkovNazarov2003}
S.~G. Bobkov and F.~L. Nazarov,
\newblock On convex bodies and log-concave probability measures with
unconditional basis,
\newblock in \emph{Geometric Aspects of Functional Analysis},
Lecture Notes in Mathematics, vol.~1807, Springer, Berlin, 2003,
pp.~53--69.
\newblock doi:10.1007/978-3-540-36428-3\_6.

\bibitem{Brenier1991}
Y.~Brenier,
\newblock Polar factorization and monotone rearrangement of vector-valued functions,
\newblock \emph{Comm. Pure Appl. Math.} \textbf{44} (1991), 375--417.
\newblock doi:10.1002/cpa.3160440402.

\bibitem{DiaconisStroock1991}
P.~Diaconis and D.~Stroock,
\newblock Geometric bounds for eigenvalues of Markov chains,
\newblock \emph{Ann. Appl. Probab.} \textbf{1} (1991), 36--61.
\newblock doi:10.1214/aoap/1177005980.

\bibitem{Fernandez2026}
V.~Manuel Fernandez,
\newblock On the $\ell_0$ isoperimetric coefficient for measurable sets,
\newblock \emph{Discrete Comput. Geom.} \textbf{75} (2026), 1378--1404.
\newblock doi:10.1007/s00454-025-00742-5.

\bibitem{Federer1969}
H.~Federer,
\newblock \emph{Geometric Measure Theory},
\newblock Die Grundlehren der mathematischen Wissenschaften, vol.~153,
Springer-Verlag, New York, 1969.

\bibitem{FigalliMaggiPratelli2010}
A.~Figalli, F.~Maggi, and A.~Pratelli,
\newblock A mass transportation approach to quantitative isoperimetric inequalities,
\newblock \emph{Invent. Math.} \textbf{182} (2010), 167--211.
\newblock doi:10.1007/s00222-010-0261-z.

\bibitem{GemanGeman1984}
S.~Geman and D.~Geman,
\newblock Stochastic relaxation, Gibbs distributions, and the Bayesian
restoration of images,
\newblock \emph{IEEE Trans. Pattern Anal. Mach. Intell.} \textbf{6} (1984),
721--741.
\newblock doi:10.1109/TPAMI.1984.4767596.

\bibitem{GoyalDeligiannidisKantas2025}
A.~Goyal, G.~Deligiannidis, and N.~Kantas,
\newblock Mixing time bounds for the Gibbs sampler under isoperimetry,
\newblock arXiv:2506.22258, 2025.

\bibitem{Harper1999}
L.~H. Harper,
\newblock On an isoperimetric problem for Hamming graphs,
\newblock \emph{Discrete Appl. Math.} \textbf{95} (1999), 285--309.

\bibitem{Harper2004}
L.~H. Harper,
\newblock \emph{Global Methods for Combinatorial Isoperimetric Problems},
\newblock Cambridge Studies in Advanced Mathematics, vol.~90,
Cambridge University Press, Cambridge, 2004.

\bibitem{JerrumSinclair1989}
M.~Jerrum and A.~Sinclair,
\newblock Approximating the permanent,
\newblock \emph{SIAM J. Comput.} \textbf{18} (1989), 1149--1178.
\newblock doi:10.1137/0218077.

\bibitem{John1948}
F.~John,
\newblock Extremum problems with inequalities as subsidiary conditions,
\newblock in \emph{Studies and Essays Presented to R.~Courant on His
60th Birthday}, Interscience, New York, 1948, pp.~187--204.

\bibitem{KannanLovaszSimonovits1995}
R.~Kannan, L.~Lov\'asz, and M.~Simonovits,
\newblock Isoperimetric problems for convex bodies and a localization lemma,
\newblock \emph{Discrete Comput. Geom.} \textbf{13} (1995), 541--559.

\bibitem{LaddhaVempala2023}
A.~Laddha and S.~S. Vempala,
\newblock Convergence of Gibbs sampling: coordinate hit-and-run mixes fast,
\newblock \emph{Discrete Comput. Geom.} \textbf{70} (2023), 406--425.
\newblock doi:10.1007/s00454-023-00497-x.

\bibitem{LoomisWhitney1949}
L.~H. Loomis and H.~Whitney,
\newblock An inequality related to the isoperimetric inequality,
\newblock \emph{Bull. Amer. Math. Soc.} \textbf{55} (1949), 961--962.

\bibitem{McCann1995}
R.~J. McCann,
\newblock Existence and uniqueness of monotone measure-preserving maps,
\newblock \emph{Duke Math. J.} \textbf{80} (1995), 309--323.
\newblock doi:10.1215/S0012-7094-95-08013-2.

\bibitem{McCann1997}
R.~J. McCann,
\newblock A convexity principle for interacting gases,
\newblock \emph{Adv. Math.} \textbf{128} (1997), 153--179.
\newblock doi:10.1006/aima.1997.1634.

\bibitem{NarayananRajaramanSrivastava2025}
H.~Narayanan, A.~Rajaraman, and P.~Srivastava,
\newblock Sampling from convex sets with a cold start using multiscale decompositions,
\newblock \emph{Probab. Theory Related Fields} \textbf{191} (2025), 1169--1232.
\newblock doi:10.1007/s00440-024-01341-w.

\bibitem{NarayananSrivastava2022}
H.~Narayanan and P.~Srivastava,
\newblock On the mixing time of coordinate hit-and-run,
\newblock \emph{Combin. Probab. Comput.} \textbf{31} (2022), 320--332.
\newblock doi:10.1017/S0963548321000328.

\bibitem{Paouris2006}
G.~Paouris,
\newblock Concentration of mass on convex bodies,
\newblock \emph{Geom. Funct. Anal.} \textbf{16} (2006), 1021--1049.
\newblock doi:10.1007/s00039-006-0584-5.

\bibitem{Sinclair1992}
A.~Sinclair,
\newblock Improved bounds for mixing rates of Markov chains and multicommodity flow,
\newblock \emph{Combin. Probab. Comput.} \textbf{1} (1992), 351--370.
\newblock doi:10.1017/S0963548300000390.

\bibitem{Turchin1971}
V.~F. Turchin,
\newblock On the computation of multidimensional integrals by the Monte Carlo
method,
\newblock \emph{Theory Probab. Appl.} \textbf{16} (1971), 720--724.
\newblock doi:10.1137/1116083.

\bibitem{LovaszKannan1999}
L.~Lov\'asz and R.~Kannan,
\newblock Faster mixing via average conductance,
\newblock in \emph{Proceedings of the 31st Annual ACM Symposium on Theory of Computing},
ACM, New York, 1999, pp.~282--287.
\newblock doi:10.1145/301250.301317.

\bibitem{Villani2003}
C.~Villani,
\newblock \emph{Topics in Optimal Transportation},
\newblock Graduate Studies in Mathematics, vol.~58, American Mathematical Society, Providence, 2003.

\bibitem{Wadia2024}
N.~S. Wadia,
\newblock A mixing time bound for Gibbs sampling from log-smooth log-concave distributions,
\newblock arXiv:2412.17899, 2024.

\bibitem{Zalinescu2002}
C.~Z\u{a}linescu,
\newblock \emph{Convex Analysis in General Vector Spaces},
\newblock World Scientific, River Edge, NJ, 2002.

\bibitem{IsakssonKindlerMossel2012}
M.~Isaksson, G.~Kindler, and E.~Mossel,
\newblock The geometry of manipulation: a quantitative proof of the
Gibbard--Satterthwaite theorem,
\newblock \emph{Combinatorica} \textbf{32} (2012), 221--250.
\newblock doi:10.1007/s00493-012-2704-1.

\bibitem{Aleksandrov1939}
A.~D. Aleksandrov,
\newblock Almost everywhere existence of the second differential of a
convex function and some properties of convex surfaces connected with it
[in Russian],
\newblock \emph{Uchen. Zap. Leningrad Gos. Univ. Math. Ser.} \textbf{6}
(1939), 3--35.

\end{thebibliography}
\end{document}